\documentclass[10pt]{amsart}
\usepackage{amssymb}
\usepackage{amscd}
\usepackage[all]{xy}

\numberwithin{equation}{section}

\def\today{\number\day\space\ifcase\month\or   January\or February\or
   March\or April\or May\or June\or   July\or August\or September\or
   October\or November\or December\fi\   \number\year}

\theoremstyle{definition}
\newtheorem{thm}{Theorem}[section]
\newtheorem{lem}[thm]{Lemma}
\newtheorem{prp}[thm]{Proposition}
\newtheorem{dfn}[thm]{Definition}
\newtheorem{cor}[thm]{Corollary}

\newtheorem{rmk}[thm]{Remark}
\newtheorem{ntn}[thm]{Notation}
\newtheorem{exa}[thm]{Example}

\newtheorem{qst}[thm]{Question}

\newcommand{\beq}{\begin{equation}}
\newcommand{\eeq}{\end{equation}}
\newcommand{\beqr}{\begin{eqnarray*}}
\newcommand{\eeqr}{\end{eqnarray*}}
\newcommand{\bal}{\begin{align*}}
\newcommand{\eal}{\end{align*}}
\newcommand{\bei}{\begin{itemize}}
\newcommand{\eei}{\end{itemize}}
\newcommand{\limi}[1]{\lim_{{#1} \to \infty}}

\newcommand{\af}{\alpha}
\newcommand{\bt}{\beta}
\newcommand{\gm}{\gamma}
\newcommand{\dt}{\delta}
\newcommand{\ep}{\varepsilon}
\newcommand{\zt}{\zeta}
\newcommand{\et}{\eta}

\newcommand{\io}{\iota}
\newcommand{\te}{\theta}
\newcommand{\ld}{\lambda}
\newcommand{\sm}{\sigma}

\newcommand{\ph}{\varphi}
\newcommand{\ps}{\psi}
\newcommand{\rh}{\rho}
\newcommand{\om}{\omega}
\newcommand{\ta}{\tau}

\newcommand{\Dt}{\Delta}

\newcommand{\Ph}{\Phi}

\newcommand{\Q}{{\mathbb{Q}}}
\newcommand{\Z}{{\mathbb{Z}}}
\newcommand{\R}{{\mathbb{R}}}
\newcommand{\C}{{\mathbb{C}}}
\newcommand{\N}{{\mathbb{Z}}_{> 0}}
\newcommand{\Nz}{{\mathbb{Z}}_{\geq 0}}

\newcommand{\spn}{{\operatorname{span}}}

\newcommand{\Isom}{{\operatorname{Isom}}}

\newcommand{\cF}{{\mathcal{F}}}

\newcommand{\andeqn}{\,\,\,\,\,\, {\mbox{and}} \,\,\,\,\,\,}

\newcommand{\Wolog}{Without loss of generality}

\newcommand{\ifo}{if and only if}

\newcommand{\ca}{C*-algebra}

\newcommand{\hm}{homomorphism}

\newcommand{\ct}{continuous}

\newcommand{\hme}{homeomorphism}

\renewcommand{\S}{\subset}
\newcommand{\ov}{\overline}
\newcommand{\SM}{\setminus}
\newcommand{\I}{\infty}

\newcommand{\Def}[1]{Definition~\ref{#1}}

\title[Reduced group Banach algebras]{Simplicity
 of reduced group Banach algebras}

\author{N.~Christopher Phillips}

\date{24~September 2019}

\address{Department of Mathematics, University  of Oregon,
       Eugene OR 97403-1222, USA.}

\email[]{ncp@darkwing.uoregon.edu}

\thanks{This material is based upon work supported by the
  US National Science Foundation under
  Grant DMS-1501144 and by the Simons Foundation Collaboration Grant
  for Mathematicians \#587103.}

\begin{document}

\begin{abstract}
Let $G$ be a discrete group.
Suppose that the reduced group C*-algebra $C^*_{\mathrm{r}} (G)$
is simple.
We use results of Kalantar-Kennedy and Haagerup,
and Banach space interpolation,
to prove that, for $p \in (1, \infty)$,
the reduced group $L^p$~operator algebra $F^p_{\mathrm{r}} (G)$
and its *-analog $B^{p, *}_{\mathrm{r}} (G)$ are simple.
If $G$ is countable, we prove that the Banach algebras
generated by the left regular representations
on reflexive Orlicz sequence spaces
and certain Lorentz sequence spaces are also simple.
We prove analogous results with simplicity
replaced by the unique trace property.
For use in the Orlicz sequence space case,
we prove that if $p \in (1, \I)$,
then any reflexive Orlicz sequence space is isomorphic
(not necessarily isometrically)
to a space gotten by interpolation
between $l^p$ and some other Orlicz sequence space.
\end{abstract}

\maketitle

\indent
We use Banach space interpolation
to show that the recent Kalantar-Kennedy and Haagerup
results on simplicity
of the reduced \ca{} of a group~$G$
also imply, for $p \in (1, \infty)$,
simplicity of its $L^p$~analog $F^p_{\mathrm{r}} (G)$
as in~\cite{Ph-Lp3},
as well as simplicity
of the $*$-algebra relative $B^{p, *}_{\mathrm{r}} (G)$
defined by Liao and Yu in~\cite{LiYu}.
(The algebra $F^1_{\mathrm{r}} (G)$ is never simple
unless $G$ has only one element.)
We further show that if $C^*_{\mathrm{r}} (G)$
has a unique tracial state,
then $F^p_{\mathrm{r}} (G)$ and $B^{p, *}_{\mathrm{r}} (G)$
each have a unique unital trace.

This paper was written in response to a question of Guoliang Yu,
about
simplicity of $F^p_{\mathrm{r}} (G)$ and $B^{p, *}_{\mathrm{r}} (G)$.
As will be seen, it is very easy to prove that
simplicity of $C^*_{\mathrm{r}} (G)$
implies simplicity of $F^p_{\mathrm{r}} (G)$,
and only slightly harder to handle $B^{p, *}_{\mathrm{r}} (G)$.
We put the problem in a general framework,
which can be used to prove simplicity
and the unique trace property for algebras
obtained from regular representations on many other Banach spaces.
We don't give a thorough investigation;
rather, we examine two kinds of examples,
Orlicz sequence spaces (as in Sections 4.a--4.c of~\cite{LndTzf1})
and Lorentz sequence spaces (as in Section~4.e of~\cite{LndTzf1}).
In both cases,
we assume for convenience that $G$ is countable.
For each of these spaces,
permutations of $\N$ define isometric operators on the space,
so we can index the sequences by $G$ instead of~$\N$,
define a left regular representation of $G$ using
isometric operators on the space,
and consider the Banach algebra generated by this representation.
For reflexive Orlicz sequence spaces,
and for the Lorentz spaces
$l^{p, r} (G)$
(analogs of the more commonly used spaces $L^{p, r} (\R^n)$)
when $1 < r < p < \I$,
we prove that simplicity of $C^*_{\mathrm{r}} (G)$
implies simplicity of the Banach algebras generated by
the left regular representations
on these spaces,
and similarly for the unique trace property.

In outline,
interpolation starts with two Banach spaces
(or, in a generality we don't use, complete quasinormed
vector spaces) $E_0$ and $E_1$
with a common dense subspace~$E$,
and constructs interpolated spaces $E_{\te}$ for $\te \in (0, 1)$.
The basic example is $E_j = L^{p_j} (X, \mu)$,
with $E_{\te} = L^{p_{\te}} (X, \mu)$ for suitable~$p_{\te}$,
determined by
\[
\frac{1}{p_{\te}} = \frac{1 - \te}{p_0} + \frac{\te}{p_1}.
\]
Moreover, if $(F_0, F_1)$ is another such pair,
with common dense subspace~$F$,
and $T \colon E \to F$ is linear
(sometimes, in generality we don't need,
satisfying weaker conditions,
such as quasilinearity on a larger space)
and extends to bounded linear operators
$T_j \colon E_j \to F_j$ for $j = 0, 1$,
then $T$ extends to bounded linear operators
$T_{\te} \colon E_{\te} \to F_{\te}$ for $\te \in (0, 1)$,
with estimates on $\| T_{\te} \|$
(different for different interpolation theorems).
For example,
in the Riesz-Thorin Interpolation Theorem
(Theorem 6.27 of~\cite{Fld2}),
one gets
\begin{equation}\label{Eq_9910_RT}
\| T_{\te} \| \leq \| T_{0} \|^{1 - \te} \| T_{1} \|^{\te}.
\end{equation}
Other interpolation theorems have different estimates.

Interpolation can be used to prove simplicity
via the condition we abstract as the Powers property
(Definition~\ref{D_8X29_Powers};
the connection with simplicity is given in
Theorem~\ref{T_8Y03_Simple}, from~\cite{Hgrp},
and Proposition~\ref{L_6Z16_PPImpSimple},
a general form of a standard argument,
originally due to Powers~\cite{Pwr}).
This property
asserts the existence of certain convex combinations
of images of the group elements which have small norm.
Since we consider representations via isometries,
such convex combinations always have norm at most~$1$.
An estimate of the form~(\ref{Eq_9910_RT})
is thus useful as long as the Powers property
holds at one endpoint.

The description above is nearly a complete proof
that simplicity of $C^*_{\mathrm{r}} (G)$ and $p \in (1, \infty)$
imply simplicity of $F^p_{\mathrm{r}} (G)$:
for example, for $p \in (1, 2)$
one interpolates the Powers property
between $p = 1$ and $p = 2$.
For $B^{p, *}_{\mathrm{r}} (G)$
one needs to work a little harder,
because the norm is more complicated.
For Orlicz sequence spaces and $l^{p, r} (G)$,
we use known interpolation theorems
in which the estimates are not quite as good as~(\ref{Eq_9910_RT})
but still good enough.
The Lorentz spaces $l^{p, r} (G)$ we use are particular examples of
Lorentz sequence spaces,
and the result should be true
for much more general Lorentz sequence spaces.
This seems to require going beyond the well known
interpolation theorems;
since our purpose is just to exhibit possibilities,
we don't investigate further.
For Orlicz sequence spaces,
we can use a well known interpolation theorem,
but we need significant work to produce a space to use in this theorem.
We prove that if $p \in (1, \I)$,
then any reflexive Orlicz sequence space is isomorphic
(not necessarily isometrically)
to a space gotten by interpolation
between $l^p$ and some other Orlicz sequence space.
The proof of this fact seems to require a direct construction.

The Orlicz sequence space result actually implies the results
for $F^p_{\mathrm{r}} (G)$,
so, in principle, the corresponding parts
of Sections \ref{Sec_Simp} and~\ref{Sec_Uniq} could be omitted.
However, the proofs given in those sections are much simpler
and are a good illustration of the general method,
and the Orlicz sequence space results
do not help with the algebras $B^{p, *}_{\mathrm{r}} (G)$.
The results on Lorentz spaces do depend on the results
for $F^p_{\mathrm{r}} (G)$.

The unique trace property is handled similarly,
using a different version of the Powers property;
see Definition~\ref{D_8Y02_gPowers}.

We do not address the reverse implications.
For example,
interpolation can be used to show that
if $F^p_{\mathrm{r}} (G)$ has the Powers property
for some $p \in (1, \I)$,
then $C^*_{\mathrm{r}} (G)$ has the Powers property
and is therefore simple.
However, we do not know whether simplicity of $F^p_{\mathrm{r}} (G)$
implies the Powers property,
and similarly for the other algebras we consider.

This paper is organized as follows.
In Section~\ref{Sec_ReducedGpAlgs}, we define
a reduced group Banach algebra for a group~$G$;
this is the general framework we use.
The definition is independent of any representation of~$G$
on a Banach space.
We then give the basic examples:
$C^*_{\mathrm{r}} (G)$,
$F^p_{\mathrm{r}} (G)$, and $B^{p, *}_{\mathrm{r}} (G)$.
In Section~\ref{Sec_Simp},
we define the Powers property for a reduced group Banach algebra,
show that it implies simplicity of the algebra,
and prove that if $C^*_{\mathrm{r}} (G)$ is simple
then so are $F^p_{\mathrm{r}} (G)$ and $B^{p, *}_{\mathrm{r}} (G)$.
Section~\ref{Sec_Uniq} is the analog for the unique trace property.
Readers interested
only in $F^p_{\mathrm{r}} (G)$ and $B^{p, *}_{\mathrm{r}} (G)$
can stop here.
In Section~\ref{Sec_Orlicz},
we introduce reduced group algebras on Orlicz sequence spaces,
and prove the analogous simplicity and unique trace results.
Most of the proof of the existence of a suitable space
to use in the interpolation argument
is postponed to Section~\ref{Sec_OrliczInt}.
Section~\ref{Sec_Other}
contains the results
on reduced group algebras on Lorentz sequence spaces.

All groups will be assumed discrete.
All Banach algebras are over~$\C$.
We will use the following terminology and notation.

\begin{dfn}\label{D_8X29_Trace}
Let $A$ be a unital Banach algebra.
A {\emph{unital trace}} on~$A$
is a \ct{} linear functional $\ta \colon A \to \C$
such that $\ta (b a) = \ta (a b)$ for all $a, b \in A$
and $\ta (1) = 1$.
\end{dfn}

A linear functional $\om$ on a unital Banach algebra
is called a state if $\| \om \| = 1$ and $\om (1) = 1$.
So a tracial state is a unital trace.
We don't need to require our traces to have norm~$1$,
only that they be bounded.

\begin{ntn}\label{N_8X29_Isometries}
Let $A$ be a unital Banach algebra.
Then we denote by $\Isom (A)$
the group of invertible isometries in~$A$,
that is,
\[
\Isom (A)
 = \bigl\{ s \in A \colon {\mbox{$s$ is invertible, $\| s \| = 1$,
      and $\| s^{-1} \| = 1$}} \bigr\}.
\]
\end{ntn}

We will use both adjoints and Banach space duality
(for operators as well as for spaces).
To distinguish them, we use the following convention.

\begin{ntn}\label{N_8Y03_Duality}
If $E$ is a Banach space,
we denote its dual by~$E'$.
If $F$ is another Banach space and $a \in L (E, F)$,
we let $a' \in L (F', E')$
be the dual (transpose) operator,
given by $a' (\om) (\xi) = \om (a \xi)$
for $\om \in F'$ and $\xi \in E$.
For $p \in [1, \I)$ and with $q \in (1, \I]$
chosen so that $\frac{1}{p} + \frac{1}{q} = 1$,
for any set $S$
we identify $l^p (S)'$ with $l^q (S)$ in the standard way.
\end{ntn}

We thank Marcin Bownik and Bill Johnson for useful answers
to the question,
``What other interesting Banach spaces of sequences are there?''
We also thank Sanaz Pooya for helpful comments
on earlier drafts of this paper.

\section{Reduced group Banach algebras}\label{Sec_ReducedGpAlgs}

In this section,
we give a general framework which covers
both $F^p_{\mathrm{r}} (G)$ and $B^{p, *}_{\mathrm{r}} (G)$,
as well as examples
constructed from Orlicz sequence and Lorentz sequence spaces.

\begin{dfn}\label{N_8X29_AbstractRedG}
Let $G$ be a group
(taken with the discrete topology).
An {\emph{(abstract) reduced group Banach algebra for~$G$}}
is a triple $(A, w, \ta)$
in which $A$ is a unital Banach algebra,
$w \colon G \to \Isom (A)$
(see Notation~\ref{N_8X29_Isometries})
is a group \hm,
and $\ta \colon A \to \C$
is a unital trace (\Def{D_8X29_Trace}),
such that the following conditions hold:
\begin{enumerate}
%
\item\label{Item_8X29_AbstractRedG_Density}
$A = {\overline{\spn}}
   \bigl( \bigl\{ w (g) \colon g \in G \bigr\} \bigr)$.
\item\label{Item_8X29_AbstractRedG_Inj}
$\ta (w (g)) = 0$ for all $g \in G \setminus \{ 1 \}$.
\item\label{Item_8X29_AbstractRedG_TrFaithful}
If $a \in A$ and $\ta (a w (g)) = 0$ for all $g \in G$,
then $a = 0$.
\end{enumerate}
\end{dfn}

Of course,
to check that $w (g) \in \Isom (A)$ for all $g \in G$,
it suffices to check that $\| w (g) \| \leq 1$ for all $g \in G$.

Condition~(\ref{Item_8X29_AbstractRedG_Density})
is a density condition,
Condition~(\ref{Item_8X29_AbstractRedG_Inj})
is a strong form of injectivity on the \hm~$w$,
and Condition~(\ref{Item_8X29_AbstractRedG_TrFaithful})
says that $\ta$ is faithful in some sense.

The standard example is as follows.

\begin{exa}\label{Ex_9217_CStar}
Let $G$ be a discrete group,
let $w$ be the standard \hm{}
from $G$ to the unitary group of $C^*_{\mathrm{r}} (G)$,
and let $\ta$ be the standard tracial state on $C^*_{\mathrm{r}} (G)$.
Then $\bigl( C^*_{\mathrm{r}} (G), w, \ta \bigr)$
is a reduced group Banach algebra for~$G$.
\end{exa}

The conditions in Definition~\ref{N_8X29_AbstractRedG} are independent.
If $G$ is a discrete group which is not amenable,
then $C^* (G)$,
with the obvious choices of~$w$ and the usual choice of~$\ta$,
satisfies (\ref{Item_8X29_AbstractRedG_Density})
and (\ref{Item_8X29_AbstractRedG_Inj})
but not~(\ref{Item_8X29_AbstractRedG_TrFaithful}).
If $G = \Z$,
$A = \C$,
$w (g) = 1$ for all $g \in G$,
and $\ta \colon A \to \C$
is the identity map,
then $(A, w, \ta)$
satisfies (\ref{Item_8X29_AbstractRedG_Density})
and (\ref{Item_8X29_AbstractRedG_TrFaithful})
but not~(\ref{Item_8X29_AbstractRedG_Inj}).
If $G$ is any ICC group,
then its group von Neumann algebra,
with the obvious choices of $w$ and~$\ta$,
satisfies (\ref{Item_8X29_AbstractRedG_Inj})
and~(\ref{Item_8X29_AbstractRedG_TrFaithful})
but not~(\ref{Item_8X29_AbstractRedG_Density}).

Injectivity of~$w$
does not imply~(\ref{Item_8X29_AbstractRedG_Inj}),
even in the presence of (\ref{Item_8X29_AbstractRedG_Density})
and~(\ref{Item_8X29_AbstractRedG_TrFaithful}):
take $G = \Z$ and $A = \C$,
fix any $\te \in \R \SM \Q$,
and define $w (n) = e^{2 \pi i n \te}$ for $n \in \Z$.

We recall the reduced group $L^p$~operator algebra
$F^p_{\mathrm{r}} (G)$ from~\cite{Ph-Lp3};
it has appeared in earlier work.
For use in interpolation arguments,
we include the case $p = \I$.

\begin{dfn}\label{D_8Y02_FprG}
Let $G$ be a group
(taken with the discrete topology).
Let $\C [G]$ be the usual complex group ring of~$G$,
and write its elements as sums $a = \sum_{g \in G} a_g u_g$
with $a_g \in \C$ for all $g \in G$ and $a_g = 0$
for all but finitely many $g \in G$.

Let $p \in [1, \I]$.
Let $w_p \colon G \to L (l^p (G))$
be the left regular representation of $G$ on $l^p (G)$.
With $\dt_{p, g} \in l^p (G)$ being the standard basis vector
corresponding to $g \in G$,
it is determined by $w_p (g) ( \dt_{p, h}) = \dt_{p, g h}$
for $g, h \in G$.
Let $\rh_p \colon \C [G] \to L (l^p (G))$
be the unital algebra \hm{}
satisfying
\[
\rh_p \Biggl( \sum_{g \in G} a_g u_g \Biggr)
 = \sum_{g \in G} a_g w_p (g)
\]
for $a = \sum_{g \in G} a_g u_g$ as above.
Define
\[
F^p_{\mathrm{r}} (G)
 = {\overline{ \rh_p ( \C [G])}}
 \S L (l^p (G)).
\]
Further define $\ta_p \colon F^p_{\mathrm{r}} (G) \to \C$
by taking $\ta_p (a)$ to be the coordinate of $a (\dt_{p, 1})$
at the identity of the group,
that is,
if $a (\dt_{p, 1}) = (\xi_g)_{g \in G}$,
then $\ta_p (a) = \xi_1$.
\end{dfn}

\begin{rmk}\label{R_9224_CStarR}
If $p = 2$ in Definition~\ref{D_8Y02_FprG},
one gets $C^*_{\mathrm{r}} (G)$
as in Example~\ref{Ex_9217_CStar}.
\end{rmk}

\begin{lem}\label{L_8Y02_FprGIsRedAlg}
Let the notation be as in Definition~\ref{D_8Y02_FprG},
with $p \in [1, \I]$.
Then $( F^p_{\mathrm{r}} (G), \, w_p, \, \ta_p)$
is a reduced group Banach algebra for~$G$.
\end{lem}

For $p \neq \I$,
the algebra $F^p_{\mathrm{r}} (G)$
is in Definition 3.3(2) of~\cite{Ph-Lp3},
with $A$ there taken to be $\C$ with the trivial action.
Most of Lemma~\ref{L_8Y02_FprGIsRedAlg},
under the additional assumption that $G$ is countable,
is then a special case of results in~\cite{Ph-Lp3}.
In particular,
existence and faithfulness of~$\ta$
is a special case of existence and faithfulness
of the standard conditional expectation
$E \colon F^p_{\mathrm{r}} (G, A, \af) \to A$,
as in Definition~4.11,
Proposition~4.8, and Proposition 4.9(1) of~\cite{Ph-Lp3}.
We give a selfcontained proof,
since we will refer to the argument later.

\begin{proof}[Proof of Lemma~\ref{L_8Y02_FprGIsRedAlg}]
It is immediate that $w_p (g)$ is an isometry for all $g \in G$.
The set $F^p_{\mathrm{r}} (G)$
is a Banach algebra because it is a closed subalgebra
of $L (l^p (G))$.

There is a bounded linear functional
$\om \colon l^p (G) \to \C$
such that if $\xi = (\xi_g)_{g \in G} \in l^p$,
then $\om (\xi) = \xi_1$;
moreover, $\| \om \| = 1$.
The functional $\ta_p$ is \ct{}
because it is given by the formula
$\ta_p (a) = \om ( a \dt_{p, 1} )$
for $a \in F^p_{\mathrm{r}} (G)$.
It is obviously unital.
A direct computation shows that
for $g, h \in G$
we have
$\ta_p \bigl( w_p (g) w_p (h) \bigr)
 = \ta_p \bigl( w_p (h) w_p (g) \bigr)$
(it is $1$ if $g = h^{- 1}$ and $0$ otherwise),
and the trace property for~$\ta_p$
follows by linearity and continuity.

Condition~(\ref{Item_8X29_AbstractRedG_Density})
of Definition~\ref{N_8X29_AbstractRedG}
holds by construction,
and Condition~(\ref{Item_8X29_AbstractRedG_Inj})
there is immediate.

For
Condition~(\ref{Item_8X29_AbstractRedG_TrFaithful})
of Definition~\ref{N_8X29_AbstractRedG},
let $a \in F^{p}_{\mathrm{r}} (G)$,
and suppose that $\ta_{p} ( a w_{p} (g) ) = 0$
for all $g \in G$.

We claim that if $h \in G$,
then $a \dt_{p, h} = 0$.
To see this,
write $a \dt_{p, h} = (\et_g)_{g \in G}$.
Then for $g \in G$ we have,
using the trace property for~$\ta_p$ at the fourth step,
\begin{align*}
\et_g
& = \om \bigl( w_{p} (g^{-1}) a \dt_{p, h} \bigr)
\\
& = \om \bigl( w_{p} (g^{-1}) a w_{p} (h) \dt_{p, 1} \bigr)
  = \ta_p \bigl( w_{p} (g^{-1}) a w_{p} (h) \bigr)
  = \ta_p \bigl( a w_{p} (h g^{-1}) \bigr)
  = 0.
\end{align*}
Since this is true for all $g \in G$,
the claim follows.

If $p \neq \I$,
the set
$\{ \dt_{p, h} \colon h \in G \}$
spans a norm dense subspace of $l^p (G)$,
so it follows that $a = 0$.

For $p = \I$,
give $l^{\I} (G)$ the weak* topology it gets
as the dual of $l^1 (G)$.
One checks that if $g \in G$ then
$w_{\I} (g)$ is the dual of the operator
$w_{1} (g^{-1}) \in L (l^1 (G))$.
Since dual operators have the same norm,
it follows that
\[
F^{\I}_{\mathrm{r}} (G)
 = \bigr\{ b' \colon b \in F^{1}_{\mathrm{r}} (G) \bigr\}.
\]
Therefore the operator $a$ above is weak* to weak* continuous.
Since $\{ \dt_{\I, h} \colon h \in G \}$
spans a weak* dense subspace of $l^{\I} (G)$,
we again get $a = 0$.
\end{proof}

The following definition is from the beginning of
Section~2 of~\cite{LiYu}.

\begin{dfn}[\cite{LiYu}]\label{D_8Y02_SAFp}
Let $G$ and $\C [G]$ be as in Definition~\ref{D_8Y02_FprG}.
Let $p \in [1, \I]$,
and let $w_p$ and $\rh_p$ be as in Definition~\ref{D_8Y02_FprG}.

For $a = \sum_{g \in G} a_g u_g \in \C [G]$
as in Definition~\ref{D_8Y02_FprG},
set $a^* = \sum_{g \in G} {\overline{a_g}} u_{g^{-1}}$.
Define an algebra norm
$\| \cdot \|_{p, *}$ on $\C [G]$
by
\[
\| a \|_{p, *}
 = \max \bigl( \| \rh_p (a) \|, \, \| \rh_p (a^*) \| \bigr)
\]
for $a \in \C [G]$.
Let $B^{p, *}_{\mathrm{r}} (G)$ be the completion
of $\C [G]$ in this norm.

We let $w_{p, *} \colon G \to \Isom ( B^{p, *}_{\mathrm{r}} (G) )$
be the composition of $g \mapsto u_g$
with the obvious map $\C [G] \to B^{p, *}_{\mathrm{r}} (G)$.
There is a contractive \hm{}
$\io_p \colon B^{p, *}_{\mathrm{r}} (G) \to F^p_{\mathrm{r}} (G)$
coming from the inequality $\| \rh_p (a) \| \leq \| a \|_{p, *}$
for $a \in \C [G]$,
and we define $\ta_{p, *} \colon B^{p, *}_{\mathrm{r}} (G) \to \C$
by $\ta_{p, *} = \ta_p \circ \io_p$.
\end{dfn}

In general,
there seems to be no reason for $B^{p, *}_{\mathrm{r}} (G)$
to be an $L^p$~operator algebra.
Proposition~2.2 of~\cite{LiYu}
shows that $B^{p, *}_{\mathrm{r}} (G)$
need not equal $F^p_{\mathrm{r}} (G)$.

\begin{lem}\label{L_8Y02_BpStarsRedAlg}
Let the notation be as in Definition~\ref{D_8Y02_SAFp},
with $p \in (1, \I)$.
Then:
\begin{enumerate}
%
\item\label{L_8Y02_BpStarsRedAlg_StarA}
$B^{p, *}_{\mathrm{r}} (G)$ is a unital Banach *-algebra.
\item\label{L_8Y02_BpStarsRedAlg_Red}
$( B^{p, *}_{\mathrm{r}} (G), \, w_{p, *}, \, \ta_{p, *})$
is a reduced group Banach algebra for~$G$.
\item\label{L_8Y02_BpStarsRedAlg_Tau}
$\ta_{p, *}$ is selfadjoint and has norm~$1$.
\end{enumerate}
\end{lem}

We won't use part~(\ref{L_8Y02_BpStarsRedAlg_Tau})
or the fact that $B^{p, *}_{\mathrm{r}} (G)$ is a *-algebra,
but these are properties one wants to have
and which are not explicit in~\cite{LiYu}.

\begin{proof}[Proof of Lemma~\ref{L_8Y02_BpStarsRedAlg}]
Throughout,
we identify $\C [G]$ with its image in $B^{p, *}_{\mathrm{r}} (G)$,
and we let
$\io_p \colon B^{p, *}_{\mathrm{r}} (G) \to F^p_{\mathrm{r}} (G)$
be as in Definition~\ref{D_8Y02_SAFp}.
Thus, for $a \in B^{p, *}_{\mathrm{r}} (G)$
we have
\begin{equation}\label{Eq_8Y03_aaStar}
\| a \|_{p, *}
 = \max \bigl( \| \io_p (a) \|, \, \| \io_p (a^*) \| \bigr).
\end{equation}

For~(\ref{L_8Y02_BpStarsRedAlg_StarA}),
one checks that $a \mapsto a^*$
is a conjugate linear multiplication reversing involution on $\C [G]$.
It follows that
$\| \rh_p ( (a b)^*) \| \leq \| \rh_p ( a^*) \| \| \rh_p ( b^*) \|$
for $a, b \in \C [G]$.
Part~(\ref{L_8Y02_BpStarsRedAlg_StarA})
now follows easily.

For~(\ref{L_8Y02_BpStarsRedAlg_Red}),
we verify the conditions of Definition~\ref{N_8X29_AbstractRedG}.
It is immediate that
$w_{p, *}$ is a group \hm.
For $g \in G$ we have $\| w_{p, *} (g) \| = 1$
because $\| w_p (g) \| = 1$ and $\| w_{p} (g^{-1}) \| = 1$.
The map $\ta_{p, *}$ is a unital trace because
$\ta_p$ is one and $\io_p$ is a \ct{} unital \hm.

Condition~(\ref{Item_8X29_AbstractRedG_Density})
of Definition~\ref{N_8X29_AbstractRedG}
holds by construction.

We claim that for $a \in B^{p, *}_{\mathrm{r}} (G)$
and $g \in G$,
we have
\begin{equation}\label{Eq_8Y03_Traces}
\ta_p \bigl( \io_p (a) w_p (g) \bigr) = \ta_{p, *} ( a w_{p, *} (g) )
\andeqn
\ta_p \bigl( \io_p (a^*) w_p (g) \bigr)
 = {\overline{\ta_{p, *} ( a w_{p, *} (g^{-1}) )}}.
\end{equation}
Both are computations when $a \in \C [G]$:
if $a = \sum_{g \in G} a_g u_g$
as in Definition~\ref{N_8X29_AbstractRedG},
the common values are $a_{g^{-1}}$ in the first case
and ${\overline{a_g}}$ in the second case.
The claim then follows by continuity.

To prove Condition~(\ref{Item_8X29_AbstractRedG_Inj})
of Definition~\ref{N_8X29_AbstractRedG},
for $g \in G \setminus \{ 1 \}$ we use the first part
of~(\ref{Eq_8Y03_Traces}) to get
$\ta_{p, *} (w_{p, *} (g)) = \ta_p (w_p (g)) = 0$.

For
Condition~(\ref{Item_8X29_AbstractRedG_TrFaithful})
of Definition~\ref{N_8X29_AbstractRedG},
let $a \in B^{p, *}_{\mathrm{r}} (G)$,
and suppose that $\ta_{p, *} ( a w_{p, *} (g) ) = 0$
for all $g \in G$.
By~(\ref{Eq_8Y03_Traces}),
we get
\[
\ta_p \bigl( \io_p (a) w_p (g) \bigr) = 0
\andeqn
\ta_p \bigl( \io_p (a^*) w_p (g) \bigr) = 0
\]
for all $g \in G$.
Lemma~\ref{L_8Y02_FprGIsRedAlg}
implies that $\ta_p$
is faithful,
so $\io_p (a) = 0$ and $\io_p (a^*) = 0$.
Therefore $a = 0$ by~(\ref{Eq_8Y03_aaStar}).

Finally, we prove~(\ref{L_8Y02_BpStarsRedAlg_Tau}).
We have $\ta_{p, *} (a^*) = {\overline{ \ta_{p, *} (a) }}$
for $a \in \C [G]$,
and selfadjointness of~$\ta_{p, *}$ follows by continuity.
We have $\| \ta_{p, *} \| \leq 1$
because $\| \io_p \| \leq 1$ and $\| \ta_p \| \leq 1$,
and $\| \ta_{p, *} \| \geq 1$ because $\ta_{p, *} (1) = 1$.
\end{proof}

\section{Simplicity of $F^p_{\mathrm{r}} (G)$
  and $B^{p, *}_{\mathrm{r}} (G)$}\label{Sec_Simp}

\indent
We define the Powers property
for a reduced group Banach algebra.
It is an abstraction of the property
Powers used in~\cite{Pwr}
to prove simplicity of $C^*_{\mathrm{r}} (F_2)$.
The unital trace doesn't appear in the definition,
so it makes formal sense
for just a pair consisting of a unital Banach algebra~$A$
and a group \hm{} $w \colon G \to \Isom (A)$.

\begin{dfn}\label{D_8X29_Powers}
Let $G$ be a group,
and let $(A, w, \ta)$
be a reduced group Banach algebra for~$G$
(Definition~\ref{N_8X29_AbstractRedG}).
We say that $(A, w, \ta)$
has the {\emph{Powers property}}
if for every finite set $S \S G \SM \{ 1 \}$
and every $\ep > 0$,
there are $n \in \N$ and $h_1, h_2, \ldots, h_n \in G$
such that for all $g \in S$ we have
\[
\Biggl\| \frac{1}{n} \sum_{j = 1}^n w (h_j g h_j^{-1}) \Biggr\| < \ep.
\]
\end{dfn}

This property is not the same as
$G$ being a Powers group
as in Definition~2.5 of~\cite{HjzPya}.

The importance of the Powers property comes from
the following result of Haagerup, based on
the Kalantar-Kennedy characterization
of groups for which $C^*_{\mathrm{r}} (G)$ is simple~\cite{KlKnn}.

\begin{thm}[Haagerup]\label{T_8Y03_Simple}
Let $G$ be a group
(taken with the discrete topology).
Suppose $C^*_{\mathrm{r}} (G)$ is simple.
Then, following the notation from Definition~\ref{D_8Y02_FprG},
$( C^*_{\mathrm{r}} (G), \, w_2, \, \ta_2)$
has the Powers property.
\end{thm}

\begin{proof}
By Remark~\ref{R_9224_CStarR},
this is the implication from (i) to~(vi)
in Theorem~4.5 of~\cite{Hgrp}.
\end{proof}

\begin{prp}\label{P_8Y07_HasPowers}
Let $G$ be a group
(taken with the discrete topology).
Suppose $C^*_{\mathrm{r}} (G)$ is simple.
Let $p \in (1, \I)$.
Then:
\begin{enumerate}
%
\item\label{Item_8Y07_Fpr}
The triple $(F^p_{\mathrm{r}} (G), \, w_p, \, \ta_p )$
of Definition~\ref{D_8Y02_FprG}
has the Powers property.
\item\label{Item_8Y07_Bpstar}
The triple $(B^{p, *}_{\mathrm{r}} (G), \, w_{p, *}, \, \ta_{p, *} )$
of Definition~\ref{D_8Y02_SAFp}
has the Powers property.
\end{enumerate}
\end{prp}

\begin{proof}
We first prove~(\ref{Item_8Y07_Fpr}).
For $p = 2$,
this is Remark~\ref{R_9224_CStarR} and Theorem~\ref{T_8Y03_Simple}.

Next,
suppose that $p \in (1, 2)$ and consider
$(F^p_{\mathrm{r}} (G), \, w_p, \, \ta_p )$.
Let $S \S G \SM \{ 1 \}$ be finite and let $\ep > 0$.
Define $\ld = 2 \big( 1 - \frac{1}{p} \big)$,
which is in $(0, 1)$.
Choose $\dt > 0$ such that
$\dt^{\ld} < \ep$.
By Theorem~\ref{T_8Y03_Simple},
there are $n \in \N$ and $h_1, h_2, ..., h_n \in G$
such that for all $g \in S$ we have
\begin{equation}\label{Eq_NormEst}
\Biggl\| \frac{1}{n} \sum_{j = 1}^n w_2 (h_j g h_j^{-1}) \Biggr\| < \dt.
\end{equation}

We apply the Riesz-Thorin Interpolation Theorem
(Theorem 6.27 of~\cite{Fld2}).
We warn that numbers $p_t$ and $q_t$ appear there,
but $q_t$ is {\emph{not}} the conjugate exponent for~$p_t$.
In the notation there,
take
$X = Y = G$,
both measures to be counting measure,
\[
p_0 = q_0 = 1,
\qquad
p_1 = q_1 = 2,
\qquad
M_0 = 1,
\andeqn
M_1 = \dt.
\]
Let $F (G)$ be the vector space of all functions from $G$ to~$\C$,
and for $k \in G$ define a linear map $w (k) \colon F (G) \to F (G)$
by $\bigl( w (k) \xi \bigr) (h) = \xi (k^{-1} h)$
for $\xi \in F (G)$ and $h \in G$.
The operator
\[
T \colon l^{1} (G) + l^{2} (G)
 \to l^{1} (G) + l^{2} (G)
\]
of Theorem 6.27 of~\cite{Fld2}
will be the restriction to $l^{1} (G) + l^{2} (G)$
of
$\frac{1}{n} \sum_{j = 1}^n w (h_j g h_j^{-1})$.
Since $w (g) |_{l^r (G)} = w_r (g)$
for all $r \in [1, \I]$,
it is clear that
$T ( l^{1} (G) ) \S l^{1} (G)$
and $T ( l^{2} (G) ) \S l^{2} (G)$,
so that
$T \bigl( l^{1} (G) + l^{2} (G) \bigr) \S l^{1} (G) + l^{2} (G)$,
as required in the Riesz-Thorin Interpolation Theorem.
Moreover,
for $\xi \in l^{1} (G)$ we trivially have
$\| T \xi \|_1 \leq M_0 \| \xi \|_1$,
and for $\xi \in l^{2} (G)$
we get from~(\ref{Eq_NormEst}) the inequality
$\| T \xi \|_2 \leq M_1 \| \xi \|_2$.
Apply Theorem 6.27 of~\cite{Fld2} with $t = \ld$,
so that the numbers $p_t$ and~$q_t$ there are
both equal to~$p$.
The conclusion is that for $\xi \in l^{p} (G)$ we have
\[
\| T \xi \|_p
 \leq M_0^{1 - \ld} M_1^{\ld} \| \xi \|_p
 = \dt^{\ld} \| \xi \|_p.
\]
So
\[
\Bigg\| \frac{1}{n} \sum_{j = 1}^n w_p (h_j g h_j^{-1}) \Bigg\|
 \leq \dt^{\ld}
 < \ep.
\]
This proves~(\ref{Item_8Y07_Fpr})
when $p \in (1, 2)$.

Now let $p \in (2, \I)$.
Let $q \in (1, 2)$
satisfy $\frac{1}{p} + \frac{1}{q} = 1$.
Let $S \S G \SM \{ 1 \}$ be finite and let $\ep > 0$.
Apply the case just done with $q$ in place of~$p$
and with $\{ g^{-1} \colon g \in S \}$ in place of~$S$,
getting $n \in \N$ and $h_1, h_2, ..., h_n \in G$
such that for all $g \in S$ we have
\[
\Biggl\| \frac{1}{n} \sum_{j = 1}^n w_q (h_j g^{-1} h_j^{-1}) \Biggr\|
 < \ep.
\]
One easily checks that the standard isomorphism
$l^p (G)' \cong l^q (G)$ (see Notation~(\ref{N_8Y03_Duality}))
identifies $w_p (h)'$ with $w_q (h^{-1})$ for all $h \in G$.
Since dual operators have the same norm,
for all $g \in S$ we get
\[
\Biggl\| \frac{1}{n} \sum_{j = 1}^n w_p (h_j g h_j^{-1}) \Biggr\|
  = \Biggl\| \Biggl[ \frac{1}{n} \sum_{j = 1}^n w_p (h_j g h_j^{-1})
        \Biggr]' \Biggr\|
  = \Biggl\| \frac{1}{n}
              \sum_{j = 1}^n w_q (h_j g^{-1} h_j^{-1}) \Biggr\|
  < \ep.
\]
This proves the case $p \in (2, \I)$,
and finishes the proof of~(\ref{Item_8Y07_Fpr}).

We now prove~(\ref{Item_8Y07_Bpstar}).
Let $S \S G \SM \{ 1 \}$ be finite and let $\ep > 0$.
Define $R = S \cup \{ g^{-1} \colon g \in S \}$,
which is a finite subset of $G \SM \{ 1 \}$.
By~(\ref{Item_8Y07_Fpr}),
there are $n \in \N$ and $h_1, h_2, ..., h_n \in G$
such that for all $g \in R$ we have
\[
\Biggl\| \frac{1}{n}
    \sum_{j = 1}^n w_p (h_j g h_j^{-1}) \Biggr\|
 < \ep.
\]
For $g \in S$ we then get, using $u_k^* = u_{k^{-1}}$ in $\C [G]$
(see Definition~\ref{D_8Y02_SAFp}) at the first step
and $g, g^{-1} \in R$ at the last step,
\[
\Biggl\| \frac{1}{n}
      \sum_{j = 1}^n w_{p, *} (h_j g h_j^{-1}) \Biggr\|
  = \max \Biggl( \Biggl\| \frac{1}{n}
        \sum_{j = 1}^n w_p (h_j g h_j^{-1}) \Biggr\|, \,
    \Biggl\|\frac{1}{n}
      \sum_{j = 1}^n w_p (h_j g^{-1} h_j^{-1})  \Biggr\| \Biggr)
  < \ep.
\]
This completes the proof.
\end{proof}

The proof of the following proposition is standard.
(It is the same as the argument in~\cite{Pwr}.
See the proofs of Lemma~6 and Theorem~1 there.)
Since it is central to this paper,
we give a full proof for completeness.

\begin{prp}\label{L_6Z16_PPImpSimple}
Let $G$ be a group
(taken with the discrete topology),
and let $(A, w, \ta)$
be a reduced group Banach algebra for~$G$
(Definition~\ref{N_8X29_AbstractRedG})
which has the Powers property
(Definition~\ref{D_8X29_Powers}).
Then $A$ is simple.
\end{prp}

\begin{proof}
Let $I \S A$ be a nonzero ideal.
We show that $I$ contains an invertible element.
Choose a nonzero element $a \in I$.
By Definition
\ref{N_8X29_AbstractRedG}(\ref{Item_8X29_AbstractRedG_TrFaithful}),
there is $k \in G$ such that $\ta (a w (k)) \neq 0$.
Define $b = \ta (a w (k))^{-1} a w (k^{-1})$.
Then $b \in I$ and $\ta (b) = 1$.
By Definition
\ref{N_8X29_AbstractRedG}(\ref{Item_8X29_AbstractRedG_Density}),
there are a finite set $S \S G$
and numbers $\ld_g \in \C$
for $g \in S$
such that
\[
\Biggl\| b - 1 - \sum_{g \in S} \ld_g w (g) \Biggr\|
 < \frac{1}{4 \| \ta \|}.
\]
\Wolog{} $1 \in S$.
Applying $\ta$ and using
Definition
\ref{N_8X29_AbstractRedG}(\ref{Item_8X29_AbstractRedG_Inj}),
we get $| \ld_1 | < \frac{1}{4}$.
Define
$c = \sum_{g \in S \setminus \{ 1 \} } \ld_g w (g)$;
the sum is now over $S \setminus \{ 1 \}$ instead of~$S$.
Then
\[
\| b - 1 - c \| < \frac{1}{2}
\andeqn
\ta (c) = 0.
\]
Define $M = 1 + \sum_{g \in S \setminus \{ 1 \} } | \ld_g |$.
Use the Powers property to choose
$n \in \N$ and $h_1, h_2, ..., h_n \in G$
such that for all $g \in S \setminus \{ 1 \}$ we have
\[
\Biggl\| \frac{1}{n} \sum_{j = 1}^n w (h_j g h_j^{-1}) \Biggr\|
 < \frac{1}{4 M}.
\]
Then
\[
\Biggl\| \frac{1}{n} \sum_{j = 1}^n w (h_j) c w (h_j^{-1}) \Biggr\|
 \leq \sum_{g \in S \setminus \{ 1 \} } | \ld_g |
     \Biggl\| \frac{1}{n} \sum_{j = 1}^n w (h_j g h_j^{-1}) \Biggr\|
 < \frac{1}{4}.
\]

Set
\[
x = \frac{1}{n} \sum_{j = 1}^n w (h_j) b w (h_j^{-1}).
\]
Then $x \in I$.
Also,
\begin{align*}
\| x - 1 \|
& \leq
   \Biggl\| \frac{1}{n} \sum_{j = 1}^n w (h_j) c w (h_j^{-1}) \Biggr\|
  + \Biggl\|
   \frac{1}{n} \sum_{j = 1}^n w (h_j) (b - 1 - c) w (h_j^{-1}) \Biggr\|
\\
& \leq
   \Biggl\| \frac{1}{n} \sum_{j = 1}^n w (h_j) c w (h_j^{-1}) \Biggr\|
      + \| b - 1 - c \|
  < \frac{1}{2} + \frac{1}{4}
  < 1.
\end{align*}
Therefore $x$ is invertible.
Since $x \in I$,
it follows that $I = A$.
\end{proof}

As a corollary, we get the following theorem.

\begin{thm}\label{T_8Y07_SimplToSimple}
Let $G$ be a group
(taken with the discrete topology).
Suppose $C^*_{\mathrm{r}} (G)$ is simple.
Let $p \in (1, \I)$.
Then:
\begin{enumerate}
%
\item\label{Item_8Y07_SimplToSimple_Fpr}
The algebra $F^p_{\mathrm{r}} (G)$
of Definition~\ref{D_8Y02_FprG}
is simple.
\item\label{Item_8Y07_SimplToSimple_Bpstar}
The algebra $B^{p, *}_{\mathrm{r}} (G)$
of Definition~\ref{D_8Y02_SAFp}
is simple.
\end{enumerate}
\end{thm}

\begin{proof}
Justified by Lemma~\ref{L_8Y02_FprGIsRedAlg}
(for~(\ref{Item_8Y07_SimplToSimple_Fpr}))
and
Lemma \ref{L_8Y02_BpStarsRedAlg}(\ref{L_8Y02_BpStarsRedAlg_Red})
(for~(\ref{Item_8Y07_SimplToSimple_Bpstar})),
combine Proposition~\ref{P_8Y07_HasPowers}
and Proposition~\ref{L_6Z16_PPImpSimple}.
\end{proof}

\begin{cor}\label{C_9903_ToSimple}
Let $G$ be a group,
taken with the discrete topology.
Suppose that there is a $G$-boundary
(in the sense of Definition~3.8 of~\cite{KlKnn})
which is a topologically free $G$-space.
Let $p \in (1, \I)$.
Then $F^p_{\mathrm{r}} (G)$ and $B^{p, *}_{\mathrm{r}} (G)$
are simple.
\end{cor}

\begin{proof}
Combine Theorem~\ref{T_8Y07_SimplToSimple}
with the implication from (5) to~(1)
in Theorem~6.2 of~\cite{KlKnn}.
\end{proof}

Theorem \ref{T_8Y07_SimplToSimple}(\ref{Item_8Y07_SimplToSimple_Fpr})
implies group algebra case
(but not the general statement for crossed products)
in Theorem 3.7 of~\cite{HjzPya}.

We also see that $L^1 (G)$ can't have the Powers property
unless $G$ is trivial.
Recall from Proposition 3.14 of~\cite{Ph-Lp3}
that $L^1 (G) = F^1_{\mathrm{r}} (G) = F^1 (G)$.

\begin{cor}\label{L1NoPowers}
Let $G$ be a group with more than one element,
taken with the discrete topology.
Then $L^1 (G)$ does not have the Powers property.
\end{cor}

\begin{proof}
We know that $L^1 (G)$ is not simple.
By Proposition 3.14 of~\cite{Ph-Lp3},
we have $L^1 (G) = F^1_{\mathrm{r}} (G)$.
Apply Lemma~\ref{L_8Y02_FprGIsRedAlg}
and Proposition~\ref{L_6Z16_PPImpSimple}.
\end{proof}

\section{Uniqueness of the trace on $F^p_{\mathrm{r}} (G)$
  and $B^{p, *}_{\mathrm{r}} (G)$}\label{Sec_Uniq}

\indent
We now define the single element Powers property.
It will be used to prove uniqueness of the trace.
As with the Powers property,
the unital trace doesn't appear in the definition,
so it makes formal sense
for just a pair consisting of a unital Banach algebra~$A$
and a group \hm{} $w \colon G \to \Isom (A)$.

\begin{dfn}\label{D_8Y02_gPowers}
Let $G$ be a group,
let $(A, w, \ta)$
be a reduced group Banach algebra for~$G$
(Definition~\ref{N_8X29_AbstractRedG}),
and let $g \in G \setminus \{ 1 \}$.
We say that $(A, w, \ta)$
has the {\emph{$g$-Powers property}}
if for every $\ep > 0$
there are $n \in \N$ and $h_1, h_2, \ldots, h_n \in G$
such that
\[
\Biggl\| \frac{1}{n} \sum_{j = 1}^n w (h_j g h_j^{-1}) \Biggr\| < \ep.
\]
\end{dfn}

It is not hard to check that
$(A, w, \ta)$ has the $g$-Powers property
\ifo{}
$0 \in {\ov{{\operatorname{Conv}}}}
  \big( \big\{ w (h g h^{-1} ) \colon h \in G \big\} \big)$.

\begin{thm}[Haagerup]\label{T_8Y03_UniqTr}
Let $G$ be a group
(taken with the discrete topology).
Suppose that $C^*_{\mathrm{r}} (G)$ has a unique tracial state.
Then, following the notation from Definition~\ref{D_8Y02_FprG},
$( C^*_{\mathrm{r}} (G), \, w_2, \, \ta_2)$
has the $g$-Powers property for every $g \in G \SM \{ 1 \}$.
\end{thm}

\begin{proof}
By Remark~\ref{R_9224_CStarR},
this is the implication from (i) to~(iv)
in Theorem~5.2 of~\cite{Hgrp}.
\end{proof}

It is clear that if $(A, w, \ta)$
has the Powers property,
then $(A, w, \ta)$
has the $g$-Powers property for every $g \in G \setminus \{ 1 \}$.
Since there are
examples of countable groups~$G$
such that $C^*_{\mathrm{r}} (G)$ has a unique tracial state
but is not simple
(Theorem~D of~\cite{LBdc}),
the converse is false.

\begin{prp}\label{P_8Y07_HasgPowers}
Let $G$ be a group
(taken with the discrete topology).
Suppose that $C^*_{\mathrm{r}} (G)$ has a unique tracial state.
Let $p \in (1, \I)$.
Then:
\begin{enumerate}
%
\item\label{Item_8Y07_HasgPowers_Fpr}
The triple $(F^p_{\mathrm{r}} (G), \, w_p, \, \ta_p )$
of Definition~\ref{D_8Y02_FprG}
has the $g$-Powers property for every $g \in G \SM \{ 1 \}$.
\item\label{Item_8Y07_HasgPowers_Bpstar}
The triple $(B^{p, *}_{\mathrm{r}} (G), \, w_{p, *}, \, \ta_{p, *} )$
of Definition~\ref{D_8Y02_SAFp}
has the $g$-Powers property for every $g \in G \SM \{ 1 \}$.
\end{enumerate}
\end{prp}

\begin{proof}
We prove~(\ref{Item_8Y07_HasgPowers_Fpr}).
For $p \in (1, 2)$,
and using Theorem~\ref{T_8Y03_UniqTr}
in place of Theorem~\ref{T_8Y03_Simple},
the proof is essentially the same as the proof of
Proposition \ref{P_8Y07_HasPowers}(\ref{Item_8Y07_Fpr}).
For $p \in (2, \I)$,
use the proof of
the corresponding case
in Proposition \ref{P_8Y07_HasPowers}(\ref{Item_8Y07_Fpr}),
but starting with the fact that,
with $\frac{1}{p} + \frac{1}{q} = 1$,
the triple $(F^q_{\mathrm{r}} (G), \, w_q, \, \ta_q )$
has the $g^{-1}$-Powers property.

Now we prove~(\ref{Item_8Y07_HasgPowers_Bpstar}).
Assume first that $p \in (1, 2)$.
Let $g \in G \SM \{ 1 \}$ and let $\ep > 0$.
Define $\ld = 2 \big( 1 - \frac{1}{p} \big)$,
which is in $(0, 1)$.
Choose $\dt > 0$ such that
$\dt^{\ld} < \ep$.
By Theorem~\ref{T_8Y03_UniqTr},
there are $n \in \N$ and $h_1, h_2, ..., h_n \in G$
such that
\begin{equation}\label{Eq_UTrNormEst}
\Biggl\| \frac{1}{n} \sum_{j = 1}^n w_2 (h_j g h_j^{-1}) \Biggr\| < \dt.
\end{equation}

Apply the Riesz-Thorin Interpolation Theorem
(Theorem 6.27 of~\cite{Fld2}),
as in the proof
of Proposition \ref{P_8Y07_HasPowers}(\ref{Item_8Y07_Fpr}).
We get
\begin{equation}\label{Eq_9226_Star}
\Bigg\| \frac{1}{n} \sum_{j = 1}^n w_p (h_j g h_j^{-1}) \Bigg\|
 < \ep.
\end{equation}
Next, in the notation of Theorem 6.27 of~\cite{Fld2},
take
$X = Y = G$,
both measures to be counting measure,
\[
p_0 = q_0 = \I,
\qquad
p_1 = q_1 = 2,
\qquad
M_0 = 1,
\andeqn
M_1 = \dt.
\]
Let $F (G)$ and $w (k) \colon F (G) \to F (G)$
be as in the proof
of Proposition \ref{P_8Y07_HasPowers}(\ref{Item_8Y07_Fpr}).
Let $T$ be as there with $l^{\I} (G)$ in place of $l^1 (G)$.
For $\xi \in l^{\I} (G)$ we trivially have
$\| T \xi \|_1 \leq M_0 \| \xi \|_1$,
and (\ref{Eq_UTrNormEst}) implies
$\| T \xi \|_2 \leq M_1 \| \xi \|_2$.
Apply Theorem 6.27 of~\cite{Fld2} with $t = \ld$,
so that the numbers $p_t$ and~$q_t$ there are
both equal to $q = \big( 1 - \frac{1}{p} \big)^{-1}$,
the conjugate exponent to~$p$.
As in the proof
of Proposition \ref{P_8Y07_HasPowers}(\ref{Item_8Y07_Fpr}),
for $\xi \in l^{q} (G)$
we get $\| T \xi \|_q \leq \dt^{\ld} \| \xi \|_q$.
So, with $u_g$ as in Definition~\ref{D_8Y02_FprG},
\begin{equation}\label{Eq_NewNewStar}
\Biggl\| \rh_q \Biggl(
    \frac{1}{n} \sum_{j = 1}^n u_{h_j g h_j^{-1}} \Biggr)
           \Biggr\|
 = \Bigg\| \frac{1}{n} \sum_{j = 1}^n w_q (h_j g h_j^{-1}) \Bigg\|
 \leq \dt^{\ld}
 < \ep.
\end{equation}

Recall that the standard isomorphism
$l^p (G)' \cong l^q (G)$ (see Notation~(\ref{N_8Y03_Duality}))
identifies $w_p (h)$ with $w_q (h^{-1})'$ for all $h \in G$.
Using this fact at the second step
and the definition of the adjoint at the first step, we get
\[
\rh_p \Biggl( \Biggl[
    \frac{1}{n} \sum_{j = 1}^n u_{h_j g h_j^{-1}} \Biggr]^* \Biggr)
  = \frac{1}{n} \sum_{j = 1}^n w_p (h_j g^{-1} h_j^{-1})
  = \Biggl[ \frac{1}{n} \sum_{j = 1}^n w_q (h_j g h_j^{-1}) \Biggr]'.
\]
Therefore,
since dual operators have the same norm,
(\ref{Eq_NewNewStar}) implies
\[
\Biggl\| \rh_p \Biggl( \Biggl[
    \frac{1}{n} \sum_{j = 1}^n u_{h_j g h_j^{-1}} \Biggr]^* \Biggr)
           \Biggr\|
 = \Biggl\| \frac{1}{n} \sum_{j = 1}^n w_q (h_j g h_j^{-1})
           \Biggr\|
 < \ep.
\]
Combining this with~(\ref{Eq_9226_Star}) gives
\begin{align*}
&\Biggl\| \frac{1}{n}
      \sum_{j = 1}^n w_{p, *} (h_j g h_j^{-1}) \Biggr\|
\\
& \hspace*{3em} {\mbox{}}
 = \max \Biggl( \Biggl\| \rh_p \Biggl(
    \frac{1}{n} \sum_{j = 1}^n u_{h_j g h_j^{-1}} \Biggr)
           \Biggr\|, \,
    \Biggl\| \rh_p \Biggl( \Biggl[
    \frac{1}{n} \sum_{j = 1}^n u_{h_j g h_j^{-1}} \Biggr]^* \Biggr)
           \Biggr\| \Biggr)
  < \ep.
\end{align*}
This proves~(\ref{Item_8Y07_HasgPowers_Bpstar})
when $p \in (1, 2)$.

The proof for $p \in (2, \I)$ is essentially the same,
exchanging $1$ and $\I$ in the argument just given.
This finishes the proof of~(\ref{Item_8Y07_HasgPowers_Bpstar}).
\end{proof}

The proof of the following lemma is the same as the argument in
the remark at the end of~\cite{Pwr}.

\begin{lem}\label{L_6Z16_PPTrZ}
Let $G$ be a group
(taken with the discrete topology),
and let $(A, w, \ta)$
be a reduced group Banach algebra for~$G$
(Definition~\ref{N_8X29_AbstractRedG}).
Let $g \in G$,
and suppose that $(A, w, \sm)$
has the $g$-Powers property
(Definition~\ref{D_8Y02_gPowers}).
Then for any continuous linear functional
$\sm \colon A \to \C$
with $\sm (b a) = \sm (a b)$ for all $a, b \in A$,
we have $\sm (w (g)) = 0$.
\end{lem}

\begin{proof}
Let $\ep > 0$;
we show that $| \sm (w (g)) | < \ep$.
By definition,
there are $n \in \N$ and $h_1, h_2, ..., h_n \in G$
such that
\[
\Bigg\| \frac{1}{n} \sum_{j = 1}^n w (h_j g h_j^{-1}) \Bigg\|
 < \frac{\ep}{\| \sm \| + 1}.
\]
For all $h \in G$, using the trace property at the second step,
we get
\[
\sm \big( w (h g h^{-1}) \big)
 = \sm \big( w (h) w (g) w (h)^{-1} \big)
 = \sm \big( w (h)^{-1} \sm (w (h) w (g) \big)
 = \sm (w (g) ).
\]
It follows that
\[
| \sm (w (g)) |
  = \Bigg| \sm \Bigg( \frac{1}{n} \sum_{j = 1}^n
          w (h_j g h_j^{-1}) \Bigg) \Bigg|
  \leq \| \sm \| \left( \frac{\ep}{\| \sm \| + 1} \right)
  < \ep.
\]
This completes the proof.
\end{proof}

\begin{cor}\label{AllgPowers}
Let $G$ be a group
(taken with the discrete topology),
and let $(A, w, \ta)$
be a reduced group Banach algebra for~$G$
(Definition~\ref{N_8X29_AbstractRedG}).
Suppose that $(A, w, \ta)$
has the $g$-Powers property
(Definition~\ref{D_8Y02_gPowers})
for all $g \in G \SM \{ 1 \}$.
Then for any continuous linear functional
$\sm \colon A \to \C$
with $\sm (b a) = \sm (a b)$ for all $a, b \in A$,
we have $\sm = \sm (1) \cdot \ta$.
\end{cor}

\begin{proof}
Since
$A = {\overline{\spn}}
   \bigl( \bigl\{ w (g) \colon g \in G \bigr\} \bigr)$,
it suffices to prove that
$\sm (w (g)) = \sm (1) \ta (w (g))$ for all $g \in G$.
Since $w (1) = 1$ and $\ta$ is a unital trace,
this is true for $g = 1$.
For $g \in G \setminus \{ 1 \}$,
we have $\sm (1) \ta (w (g)) = 0$
by
Definition \ref{N_8X29_AbstractRedG}(\ref{Item_8X29_AbstractRedG_Inj})
and $\sm (w (g)) = 0$ by Lemma~\ref{L_6Z16_PPTrZ}.
\end{proof}

As a corollary, we get the following theorem.

\begin{thm}\label{T_8Y07_UniqToUniq}
Let $G$ be a group
(taken with the discrete topology).
Suppose that $C^*_{\mathrm{r}} (G)$ has a unique tracial state.
Let $p \in (1, \I)$.
Then:
\begin{enumerate}
%
\item\label{Item_8Y07_UniqToUniq_Fpr}
The algebra $F^p_{\mathrm{r}} (G)$
of Definition~\ref{D_8Y02_FprG}
has a unique unital trace.
\item\label{Item_8Y07_UniqToUniq_Bpstar}
The algebra $B^{p, *}_{\mathrm{r}} (G)$
of Definition~\ref{D_8Y02_SAFp}
has a unique unital trace.
\end{enumerate}
\end{thm}

\begin{proof}
Justified by Lemma~\ref{L_8Y02_FprGIsRedAlg}
(for~(\ref{Item_8Y07_UniqToUniq_Fpr}))
and
Lemma \ref{L_8Y02_BpStarsRedAlg}(\ref{L_8Y02_BpStarsRedAlg_Red})
(for~(\ref{Item_8Y07_UniqToUniq_Bpstar})),
combine Proposition~\ref{P_8Y07_HasgPowers}
and Corollary~\ref{AllgPowers}.
\end{proof}

Theorem \ref{T_8Y07_UniqToUniq}(\ref{Item_8Y07_UniqToUniq_Fpr})
implies group algebra case
(but not the general statement for crossed products)
in Theorem 3.5 of~\cite{HjzPya}.

\begin{cor}\label{C_9903_ForUniqTr}
Let $G$ be a group,
taken with the discrete topology.
Suppose that $G$ has no nontrivial amenable normal subgroups.
Let $p \in (1, \I)$.
Then each of $F^p_{\mathrm{r}} (G)$ and $B^{p, *}_{\mathrm{r}} (G)$
has a unique unital trace.
\end{cor}

\begin{proof}
Recall (see, for example,
before Proposition~2.8 of~\cite{BrKlKnOz})
that the amenable radical of~$G$
is the largest amenable normal subgroup of~$G$.
Now combine Theorem~\ref{T_8Y07_UniqToUniq}
with Corollary~4.3 of~\cite{BrKlKnOz}.
\end{proof}

\section{Orlicz functions}\label{Sec_Orlicz}

Orlicz sequence spaces, as described in
Sections 4.c.1--4.c.3 of~\cite{LndTzf1},
are a generalization of $l^p$~spaces for $p \in [1, \I]$.
They are sufficiently symmetric to support a regular representation
of a countable group.
In this section,
we show that simplicity of $C^*_{\mathrm{r}} (G)$
implies simplicity of the analogous algebra defined on any
reflexive Orlicz sequence space,
and similarly for the unique trace property.
Orlicz sequence spaces come in great variety;
see, for example, Section 4.c.3 of~\cite{LndTzf1}.
We mention just one type of example;
the facts about it are gotten by combining, in~\cite{LndTzf1},
Examples 4.c.6 and 4.c.7, Theorem 4.a.9,
and the remark after Proposition 4.b.3,
and one must look at the constructions to see that
arbitrary values of $p_0$ and~$p_1$ can occur.
For any $p_0, p_1 \in (1, \I)$
with $p_0 \leq p_1$,
there is an Orlicz sequence space~$E$
which has subspaces isomorphic to~$l^p$
exactly when $p_0 \leq p \leq p_1$,
does not have complemented subspaces isomorphic to~$l^p$
for any~$p$,
and such that every bounded linear map from $E$ to $l^p$
is compact if $p < p_0$
and every bounded linear map from $l^p$ to $E$
is compact if $p > p_1$.

Since we want a regular representation of the group~$G$,
we index our sequences by~$G$ rather than by~$\N$.
When $G$ is countable,
the spaces in the following definition are exactly those
at the beginning of Section~4.a of~\cite{LndTzf1},
where they are called $l_M$ and $h_M$.

\begin{dfn}\label{D_9226_OrliczGroupSpace}
Let $M \colon [0, \I) \to [0, \I)$
be \ct, nondecreasing, convex,
and satisfy $M (0) = 0$,
$M (t) > 0$ for $t > 0$,
and $\limi{t} M (t) = \I$.
(Such a function is called a
nondegenerate Orlicz function in Definition 4.a.1 of~\cite{LndTzf1}.)
Let $G$ be a countable group,
taken with the discrete topology.
For any family $\xi = (\xi_g)_{g \in G}$ of complex numbers,
define
\[
\| \xi \|_{M}
 = \inf \Biggl( \Biggl\{ \rh > 0 \colon
    \sum_{g \in G} M \bigl( \rh^{-1} | \xi_{g} | \bigr) \leq 1
           \Biggr\} \Biggr).
\]
Then we let $l^M (G)$,
the {\emph{Orlicz space of~$G$}}
(with parameter~$M$),
be the Banach space consisting
of all $\xi$ such that $\| \xi \|_{M} < \I$,
with the norm $\| \cdot \|_{M}$.

For $g \in G$,
let $\dt_{M, g} \in l^M (G)$ be the function
$\dt_{M, g} (g) = 1$
and $\dt_{M, g} (h) = 0$ for $h \in G \SM \{ g \}$.
We define $h^M (G)$ to be the closed linear span
\[
h^M (G) = {\overline{\spn}}
   \bigl( \bigl\{ \dt_{M, g} \colon g \in G \bigr\} \bigr).
\]
(This is not the definition in~\cite{LndTzf1},
but is equivalent to it by Proposition 4.a.2 of~\cite{LndTzf1}.)
The space~$h^M$ space may or may not be equal to $l^M (G)$.

We define the {\emph{left regular representation}}
of $G$ on $l^M (G)$
to be the function $w_{M, G} \colon G \to L ( l^M (G) )$
given by, for $\xi = (\xi_g)_{g \in G} \in l^M (G)$
and $h \in G$,
\[
\bigl( w_{M, G} (g) \xi \bigr)_h = \xi_{g^{-1} h},
\]
and we define the {\emph{left regular representation}}
of $G$ on $h^M (G)$ to be the
function $g \mapsto w_{M, G} (g) |_{h^M (G)} \in L (h^M (G))$.
\end{dfn}

A general Orlicz function
is not required to satisfy $M (t) > 0$ for $t > 0$,
and need not be strictly increasing.

\begin{ntn}\label{N_9817_F}
We denote by $\cF$ the set of all
\ct{} strictly increasing bijections from $[0, \I)$ to $[0, \I)$.
\end{ntn}

\begin{lem}\label{L_9817_FProp}
Let $\cF$ be as in Notation~\ref{N_9817_F}.
Then:
\begin{enumerate}
%
\item\label{Item_L_9817_FProp_Limi}
If $M \in \cF$ then $\limi{t} M (t) = \I$.
\item\label{Item_L_9817_FProp_Inv}
Every function in $\cF$
is invertible,
and its inverse is in~$\cF$.
\item\label{Item_L_9817_FProp_Orl}
A function $M \colon [0, \I) \to [0, \I)$
is a nondegenerate Orlicz function \ifo{}
$M \in \cF$ and $M$ is convex.
\item\label{Item_L_9817_FProp_InvOrl}
A function $\ph \colon [0, \I) \to [0, \I)$
is the inverse of a nondegenerate Orlicz function \ifo{}
$\ph \in \cF$ and $\ph$ is concave.
\end{enumerate}
\end{lem}

\begin{proof}
Parts (\ref{Item_L_9817_FProp_Limi})
and~(\ref{Item_L_9817_FProp_Inv}) are immediate.
A nondegenerate Orlicz function $M$ as in
Definition~\ref{D_9226_OrliczGroupSpace}
is strictly increasing,
by convexity and since $M (t) > 0$ for $t > 0$.
Convexity further implies that $\limi{t} M (t) = \I$.
This is~(\ref{Item_L_9817_FProp_Orl}).
Part~(\ref{Item_L_9817_FProp_InvOrl}) is now immediate.
\end{proof}

\begin{prp}\label{P_9226_OrliczRedGpAlg}
Adopt the notation of Definition~\ref{D_9226_OrliczGroupSpace}.
Let $F^M_{\mathrm{r}} (G) \S L ( h^M (G) )$ be the closed linear span
\[
F^M_{\mathrm{r}} (G) = {\overline{\spn}}
   \bigl( \bigl\{ w_{M, G} (g) |_{h^M (G)}
      \colon g \in G \bigr\} \bigr).
\]
Define $\ta_{M, G} \colon F^M_{\mathrm{r}} (G) \to \C$
by as follows:
if $a (\dt_{M, 1}) = (\xi_g)_{g \in G}$,
then $\ta_{M, G} (a) = M^{-1} (1) \xi_1$.
Then
$\bigl( F^M_{\mathrm{r}} (G),
   \, w_{M, G} (\cdot) |_{h^M (G)}, \, \ta_{M, G} \bigr)$
is a reduced group Banach algebra for~$G$.
\end{prp}

\begin{proof}
Using Lemma \ref{L_9817_FProp}(\ref{Item_L_9817_FProp_InvOrl}),
one sees that $M^{-1} (1)$ exists and is in $(0, \I)$,
and one then easily checks that that
$\| \dt_{M, g} \| = M^{-1} (1)^{-1}$ for $g \in G$.

The proof is now essentially the same
as that of the case $p \in [1, \I)$ of Lemma~\ref{L_8Y02_FprGIsRedAlg}.
We need to know that $w_{M, G} (g)$
is an isometry for $g \in G$,
which is immediate,
and that the formula
$\om \bigl( (\xi_g)_{g \in G} \bigr) = M^{-1} (1) \xi_1$
defines a linear functional $\om \colon l^M \to \C$
with $\| \om \| = 1$,
which is easy to check.
\end{proof}

\begin{rmk}\label{R_9915_OrGenLp}
In Definition~\ref{D_9226_OrliczGroupSpace},
one readily checks that for $t \in [1, \I)$
the function $M (t) = t^p$ is a nondegenerate Orlicz function,
and that the norm $\| \cdot \|_M$ is just $\| \cdot \|_p$,
so that $h^M = l^M = l^p$
and $F^M_{\mathrm{r}} (G) = F^p_{\mathrm{r}} (G)$.
\end{rmk}

We will prove that
if $h^M (G)$ is reflexive and $C^*_{\mathrm{r}} (G)$ is simple,
then $F^M_{\mathrm{r}} (G)$ is simple,
by expressing $h^M (G)$ as an interpolation space
between $l^2 (G)$ and some other Orlicz sequence space.
Some work is needed to construct such an Orlicz sequence space.

For nondegenerate Orlicz functions,
the following definition is contained in
Proposition 4.a.5(iii) of~\cite{LndTzf1}.

\begin{dfn}\label{D_9817_EqAtZ}
Let $M, N \in \cF$.
Then {\emph{$M$ is equivalent to $N$ at zero}},
written $M \sim N$,
if there exist $c_0, c_1, t_0 \in (0, \I)$
such that for all $t \in [0, t_0]$ we have
$c_0^{-1} N (c_1^{-1} t) \leq M (t) \leq c_0 N (c_1 t)$.
\end{dfn}

The importance of this condition is explained by the following
result.

\begin{prp}[Proposition 4.a.5 of~\cite{LndTzf1}]\label{P_9817_SameSp}
Let $M$ and $N$ be nondegenerate Orlicz functions.
Then $M \sim N$ \ifo{}
$l^M$ and $l^N$ contain exactly the same sequences
and the identity map from $l^M$ to $l^N$ is a \hme.
\end{prp}

\begin{prp}\label{P_9825_IntermSpace}
Let $M$ be a nondegenerate Orlicz function
and suppose that $l^M$ is reflexive.
Let $p \in (1, \I)$.
Then there are $\te \in (0, 1)$
and a nondegenerate Orlicz function~$N$
such that $x \mapsto (x^{1/p})^{1 - \te} N^{- 1} (x)^{\te}$
is the inverse function of an Orlicz function
which is equivalent at zero to~$M$
and such that $h^N = l^N$.
\end{prp}

We postpone the proof to Section~\ref{Sec_OrliczInt}.
We need only one value of~$p$.

The following interpolation result
is a special case of Theorem~1 (in Section~3) of~\cite{RaoM},
with terminology replaced by that of~\cite{LndTzf1}.

\begin{thm}\label{T_9904_RaoInterp}
Let $N_0, N_1 \in \cF$ be Orlicz functions.
For $\te \in (0, 1)$,
let $N_{\te}$ be the function determined by
$N_{\te}^{- 1} (t) = N_0^{- 1} (t)^{1 - \te} N_1^{- 1} (t)^{\te}$
for $t \in [0, \I)$.
Then $N_{\te}$ is an Orlicz function.
Further assume that $h^{N_{j}} = l^{N_{j}}$ for $j \in \{ 0, 1 \}$.
Let $F$ denote the vector space of all functions from $\N$ to~$\C$
with finite support, regarded as a subspace of $l^{N_{\te}}$
for $\te \in [0, 1]$.
Then there are norms $\| \cdot \|_{N_{\te}}^*$
and $\| \cdot \|_{N_{\te}}'$
on the spaces $l^{N_{\te}}$,
each equivalent to $\| \cdot \|_{N_{\te}}$,
such that, whenever
$T \colon F \to F$ is a linear operator
and $C_0, C_1 \in [0, \I)$ are constants such that
$\| T \xi \|_{N_{j}}^* \leq C_{j} \| \xi \|_{N_{j}}^*$
for all $\xi \in F$ and $j \in \{ 0, 1 \}$,
then
$\| T \xi \|_{N_{\te}}'
 \leq C_{0}^{1 - \te} C_{1}^{\te} \| \xi \|_{N_{\te}}^*$
for all $\xi \in F$.
\end{thm}

The conclusion
uses different norms on $\xi$ and $T \xi$,
even though they are both in~$l^{N_{\te}}$.
There is an implied norm on the algebra $L ( l^{N_{\te}} )$.
Presumably it is not a Banach algebra norm,
although it is equivalent to the norms on $L ( l^{N_{\te}} )$
gotten using any of $\| \cdot \|_{N_{\te}}$,  $\| \cdot \|_{N_{\te}}^*$,
or $\| \cdot \|_{N_{\te}}'$.
There is more on $\| \cdot \|_{N_{\te}}^*$
and $\| \cdot \|_{N_{\te}}'$
in~\cite{RaoM},
but we do not need this information for our purposes.

\begin{proof}[Proof of Theorem \ref{T_9904_RaoInterp}]
We apply the conclusion in (3.1) in Theorem~1 (in Section~3)
of~\cite{RaoM},
with both measure spaces being $\N$ with counting measure,
and with, in the notation there,
$Q_0 = \Ph_0 = N_0$ and $Q_1 = \Ph_1 = N_1$.
We need to match the terminology.
A nondegenerate Orlicz function $M$ in~\cite{LndTzf1}
is the restriction to $[0, \I)$
of the continuous Young's function~$\Ph$,
as at the the beginning of Section~2 of~\cite{RaoM},
given  by $\Ph (t) = M (| t |)$.
The space $L^{\Ph} (\N)$ in~\cite{RaoM} is our $l^M$;
see the beginning of Section~2 of~\cite{RaoM}.
Our $\| \cdot \|_M$ is called $N_{\Ph}$ in~\cite{RaoM}
(see (2.2) there),
and the norm $\| \cdot \|_{\Ph}$ used in~\cite{RaoM}
(see (2.3) there)
is equivalent to ours by (2.4) of~\cite{RaoM}.
Our space $h^M$ is ${\mathfrak{M}}^{\Ph}$ in~\cite{RaoM};
see Definition~1 (in Section~2) of~\cite{RaoM}.
The notation $Q_s^{+}$ in~(3.1) of~\cite{RaoM}
is defined in the remark on page 547 there,
and $\| \cdot \|_{Q_s^{+}}$
is equivalent to our $\| \cdot \|_{N_{\te}}$
by, in~\cite{RaoM}, combining (2.4) with Lemma~4 (in Section~2).
\end{proof}

We use Theorem \ref{T_9904_RaoInterp} in the following form.

\begin{cor}\label{C_9904_InterpForUse}
Let the notation and hypotheses be as in Theorem~\ref{T_9904_RaoInterp}.
Let $M \in \cF$ be an Orlicz function such that $N_{\te} \sim M$
(Definition~\ref{D_9817_EqAtZ}).
Then there is $K \in [0, \I)$
such that whenever
$T \colon F \to F$ is a linear operator
and $C_0, C_1 \in [0, \I)$ are constants such that
$\| T \xi \|_{N_j} \leq C_{j} \| \xi \|_{N_j}$
for all $\xi \in F$ and $j \in \{ 0, 1 \}$,
then
$\| T \xi \|_{M} \leq K C_{0}^{1 - \te} C_{1}^{\te} \| \xi \|_{M}$
for all $\xi \in F$.
\end{cor}

\begin{proof}
Using the statement and notation of Theorem \ref{T_9904_RaoInterp},
and also applying Proposition~\ref{P_9817_SameSp},
there exist $R, R_0 \in (0, \I)$
such that for all $\xi \in F$ we have
\[
\| \xi \|_{N_{\te}} \leq R_0 \| \xi \|_{N_{\te}}'
\andeqn
R^{-1} \| \xi \|_M \leq \| \xi \|_{N_{\te}}^* \leq R \| \xi \|_M.
\]
For $\xi \in F$ we then have
\[
\| T \xi \|_M
 \leq R \| T \xi \|_{N_{\te}}^*
 \leq R R_0 \| T \xi \|_{N_{\te}}'
 \leq R R_0 C_{0}^{1 - \te} C_{1}^{\te} \| \xi \|_{N_{\te}}^*
 \leq R^2 R_0 C_{0}^{1 - \te} C_{1}^{\te} \| \xi \|_{M},
\]
which is the conclusion with $K = R^2 R_0$.
\end{proof}

\begin{thm}\label{T_9905_OrlAlg}
Let $G$ be a countable group,
taken with the discrete topology,
and let $M$ be a nondegenerate Orlicz function
such that $l^M$ is reflexive.
Let
$\bigl( F^M_{\mathrm{r}} (G), w_{M, G}, \ta_{M, G} \bigr)$
be as in Proposition~\ref{P_9226_OrliczRedGpAlg}
and let $( C^*_{\mathrm{r}} (G), \, w_2, \, \ta_2)$
be as in Example~\ref{Ex_9217_CStar}.
\begin{enumerate}
%
\item\label{Item_T_9905_OrlAlg_Powers}
If $( C^*_{\mathrm{r}} (G), \, w_2, \, \ta_2)$
has the Powers property,
then $\bigl( F^M_{\mathrm{r}} (G), w_{M, G}, \ta_{M, G} \bigr)$
has the Powers property.
\item\label{Item_T_9905_OrlAlg_OneElt}
Let $g \in G$,
and suppose that $( C^*_{\mathrm{r}} (G), \, w_2, \, \ta_2)$
has the $g$-Powers property.
Then $\bigl( F^M_{\mathrm{r}} (G), w_{M, G}, \ta_{M, G} \bigr)$
has the $g$-Powers property.
\end{enumerate}
\end{thm}

Presumably countability of~$G$ is not needed.
Since our purpose is to illustrate what can be done,
we do not investigate the uncountable case here.

\begin{proof}[Proof of Theorem~\ref{T_9905_OrlAlg}]
Recall from Proposition 4.b.2 and Proposition 4.a.4 of~\cite{LndTzf1}
that if $M$ is a nondegenerate Orlicz function
such that $l^M$ is reflexive,
then $h^M = l^M$.
Therefore we work in~$l^M$.

Use any bijection from $G$ to~$\N$
to identify $l^M (G)$ isometrically with~$l^M$.
Apply Proposition~\ref{P_9825_IntermSpace}
with $M$ as given and
with $p = 2$,
getting $\te \in (0, 1)$ and an Orlicz function
which we call~$N_1$.
Set $N_0 (t) = t^2$ for $t \in [0, \I)$.
Then $h^{M_0} = l^{M_0}$ and $h^{M_1} = l^{M_1}$.
With these choices,
in the notation of Corollary~\ref{C_9904_InterpForUse},
we have $M \sim N_{\te}$,
so let $K$ be the constant there.

For~(\ref{Item_T_9905_OrlAlg_Powers}),
let $\ep > 0$ and let $S \S G$ be a finite set.
Choose $\dt > 0$ such that $\dt < ( \ep / K)^{1 / (1 - \te)}$.
Applying Definition~\ref{D_8X29_Powers},
choose $n \in \N$ and $h_1, h_2, \ldots, h_n \in G$
such that for all $g \in S$ we have
\[
\Biggl\| \frac{1}{n} \sum_{j = 1}^n w_2 (h_j g h_j^{-1}) \Biggr\| < \dt.
\]
Apply the estimate in Corollary~\ref{C_9904_InterpForUse}
with $T$ being the restriction of
$\frac{1}{n} \sum_{j = 1}^n w_{M, G} (h_j g h_j^{-1})$
to the subspace $F$ consisting of sequences with finite support,
with $C_0 = \dt$, and with $C_1 = 1$.
Then use density of~$F$ to extend by continuity.
The result is
\[
\Biggl\| \frac{1}{n} \sum_{j = 1}^n w_{M, G} (h_j g h_j^{-1}) \Biggr\|
 \leq K \dt^{1 - \te}
 < \ep.
\]
The proof of~(\ref{Item_T_9905_OrlAlg_Powers})
is complete.

The proof of~(\ref{Item_T_9905_OrlAlg_OneElt})
is the same,
using Definition~\ref{D_8Y02_gPowers}
in place of Definition~\ref{D_8X29_Powers},
and using the same choice of~$\dt$.
\end{proof}

As a corollary, we get the following theorem.

\begin{thm}\label{T_9905_OrlSimplToSimple}
Let $G$ be a countable group,
taken with the discrete topology.
Let $M$ be a nondegenerate Orlicz function
such that $l^M$ is reflexive,
and let $F^M_{\mathrm{r}} (G)$
be as in Proposition~\ref{P_9226_OrliczRedGpAlg}.
Then:
\begin{enumerate}
%
\item\label{Item_T_9905_OrlSimplToSimple}
If $C^*_{\mathrm{r}} (G)$ is simple,
then $F^M_{\mathrm{r}} (G)$ is simple.
\item\label{Item_T_9905_OrlUniqToUniq}
If $C^*_{\mathrm{r}} (G)$ has a unique tracial state,
then $F^M_{\mathrm{r}} (G)$ has a unique unital trace.
\end{enumerate}
\end{thm}

\begin{proof}
For~(\ref{Item_T_9905_OrlSimplToSimple}),
combine Theorem~\ref{T_8Y03_Simple},
Theorem \ref{T_9905_OrlAlg}(\ref{Item_T_9905_OrlAlg_Powers}),
and Proposition~\ref{L_6Z16_PPImpSimple}.
For~(\ref{Item_T_9905_OrlUniqToUniq}),
combine Theorem~\ref{T_8Y03_UniqTr},
Theorem \ref{T_9905_OrlAlg}(\ref{Item_T_9905_OrlAlg_OneElt}),
and Corollary~\ref{AllgPowers}.
\end{proof}

\begin{qst}\label{Q_9916_IfNotRefl}
Let $G$ be any countable group,
taken with the discrete topology.
Let $M$ be a nondegenerate Orlicz function
such that $l^M$ is not reflexive.
Does it follow that the algebra $F^M_{\mathrm{r}} (G)$ is not simple?
\end{qst}

As evidence for a positive answer,
it follows from
Proposition 4.b.2 and Proposition 4.a.4 of~\cite{LndTzf1}
that if $l^M$ is not reflexive
then $l^M$ contains a copy of~$l^{\I}$.
We also point out the duality relations
in Proposition 4.b.1 of~\cite{LndTzf1}.

\section{The proof of
 Proposition~\ref{P_9825_IntermSpace}}\label{Sec_OrliczInt}

\begin{lem}\label{L_9817_Inv}
Let $M, N \in \cF$.
Then $M \sim N$ \ifo{} $M^{-1} \sim N^{-1}$.
\end{lem}

\begin{proof}
It suffices to prove
that $M \sim N$ implies $M^{-1} \sim N^{-1}$.
So assume $M \sim N$, and let $c_0, c_1, t_0 \in (0, \I)$
be as in Definition~\ref{D_9817_EqAtZ}.
Set $\ph = M^{-1}$ and $\ps = N^{-1}$.
Set $x_0 = \min \bigl( N (c_1 t_0), \, N (c_1^{-1} t_0) \bigr)$.
Then for $x \in [0, x_0]$ put
$t = c_1^{-1} \ps (x)$, observe that $t \in [0, t_0]$,
and apply $\ph$ to both sides of the inequality
$M (t) \leq c_0 N (c_1 t)$,
getting $c_1^{-1} \ps (x) \leq \ph (c_0 x)$,
which is $\ps (x) \leq c_1 \ph (c_0 x)$.
Similarly, put $t = c_1 \ps (x)$, observe that $t \in [0, t_0]$,
and apply $\ph$ to both sides of the inequality
$c_0^{-1} N (c_1^{-1} t) \leq M (t)$,
getting $\ph (c_0^{-1} x) \leq c_1 \ps (x)$,
which is $c_1^{-1} \ph (c_0^{-1} x) \leq \ps (x)$.
\end{proof}

\begin{ntn}\label{D_9817_Exps}
Let $M \in \cF$.
For $t_0 \in (0, 1]$,
set
\[
S_0 (M, t_0)
 = \left\{ r \in (0, \I) \colon
   \sup_{\ld \in (0, 1], \, t \in (0, t_0]}
          \frac{M (\ld t)}{M (\ld) t^r} < \I \right\}
\]
and
\[
S_1 (M, t_0)
 = \left\{ r \in (0, \I) \colon
   \inf_{\ld \in (0, 1], \, t \in (0, t_0]}
          \frac{M (\ld t)}{M (\ld) t^r} > 0 \right\}.
\]
Define
\[
\af_{0} (M, t_0)
 = \sup (S_0 (M, t_0))
\andeqn
\af_{1} (M, t_0)
 = \inf ( S_1 (M, t_0) ),
\]
and define
\[
\af_{0} (M) = \sup_{t_0 \in (0, 1]} \af_{0} (M, t_0)
\andeqn
\af_{1} (M) = \inf_{t_0 \in (0, 1]} \af_{1} (M, t_0).
\]
\end{ntn}

The numbers $\af_{0} (M, 1)$ and $\af_{1} (M, 1)$
are called $\af_M$ and $\bt_M$ in Theorem 4.a.9 of~\cite{LndTzf1}.
We will need the quantities defined here in Lemma~\ref{L_9819_Exp1}.

\begin{rmk}\label{R_9821_Interval}
Let $p, r \in (0, \I)$ satisfy $p < r$.
It is easy to see that,
following Notation~\ref{D_9817_Exps},
$r \in S_0 (M, t_0)$ implies $p \in S_0 (M, t_0)$
and $p \in S_1 (M, t_0)$ implies $r \in S_1 (M, t_0)$.
\end{rmk}

\begin{lem}\label{L_9818_OrderExp}
Adopt Notation~\ref{D_9817_Exps}.
Let $M \in \cF$, and let $t_0, t_1$ satisfy
$0 < t_1 \leq t_0 \leq 1$.
Then
\[
\af_{0} (M, t_0)
 \leq \af_{0} (M, t_1)
 \leq \af_{0} (M)
 \leq \af_{1} (M)
 \leq \af_{1} (M, t_1)
 \leq \af_{1} (M, t_0).
\]
\end{lem}

\begin{proof}
We need only prove that
if $t_0 \in (0, 1]$
then $\af_{0} (M, t_0) \leq \af_{1} (M, t_0)$.

Let $r \in (0, \I)$
satisfy $r > \af_{1} (M, t_0)$.
Then there is $p < r$ such that
\[
c = \inf_{\ld \in (0, 1], \, t \in (0, t_0]}
          \frac{M (\ld t)}{M (\ld) t^p}
\]
satisfies $c > 0$.
For $\ld \in (0, 1]$ and $t \in (0, t_0]$,
we have
\[
c t^{p - r} \leq \frac{M (\ld t)}{M (\ld) t^p} \cdot t^{p - r}
  = \frac{M (\ld t)}{M (\ld) t^r}.
\]
Since $p - r < 0$,
we have $\sup_{t \in (0, t_0]} c t^{p - r} = \I$.
So $r \not\in S_{0} (M, t_0)$.
Thus $r \geq \af_{0} (M, t_0)$ by Remark~\ref{R_9821_Interval}.
The result follows.
\end{proof}

\begin{prp}\label{P_9818_IfRefl}
Let $M$ be a nondegenerate Orlicz function,
and suppose that $l^M$ is reflexive.
Then there exists a nondegenerate Orlicz function $N$
such that $N \sim M$
and $1 < \af_{0} (N) \leq \af_{1} (N) < \I$.
\end{prp}

\begin{proof}
Combine Proposition 4.c.8 of~\cite{LndTzf1},
Proposition 4.c.4 of~\cite{LndTzf1}
(see Theorem 4.a.9 of~\cite{LndTzf1} for the notation),
and the case $t_0 = 1$ of Lemma~\ref{L_9818_OrderExp}.
\end{proof}

\begin{lem}\label{L_9819_InvExp}
Let $M \in \cF$ satisfy $M (1) = 1$.
Then, following Notation~\ref{D_9817_Exps},
$\af_{0} (M^{-1}) = \af_{1} (M)^{-1}$
and $\af_{1} (M^{-1}) = \af_{0} (M)^{-1}$.
\end{lem}

\begin{proof}
The first equation follows from the second for $M^{-1}$,
so we only prove the second.

We first claim that $\af_{0} (M) \leq \af_{1} (M^{-1})^{-1}$.
Suppose $\bt \in (0, \I)$ and $\bt < \af_{0} (M)$.
Then there is $t_0 \in (0, 1]$ such that $\bt < \af_{0} (M, t_0)$.
By Remark~\ref{R_9821_Interval},
there is $c \in (0, \I)$
such that for all $\ld \in (0, 1]$ and $t \in (0, t_0]$,
we have
\begin{equation}\label{Eq_9821_Bound}
M (\ld t) \leq c M (\ld) t^{\bt}.
\end{equation}
Define
$x_0 = \min \bigl( 1, \, c t_0^{\bt} \bigr)$.
Let $\et \in (0, 1]$ and let $x \in (0, x_0]$.
Set $\ld = M^{-1} (\et)$
and $t = c^{- 1 / \bt} x^{1 / \bt}$.
Then $\ld \in (0, 1]$ and $t \in (0, t_0]$,
so we can apply (\ref{Eq_9821_Bound}),
getting
\[
M \bigl( M^{-1} (\et)
   \cdot c^{- 1 / \bt} x^{1 / \bt} \bigr)
 \leq c M \bigl( M^{-1} (\et) \bigr)
   \bigl[ c^{- 1 / \bt} x^{1 / \bt} \bigr]^{\bt}
 = \et x.
\]
Since $M^{-1}$ is strictly increasing,
we get
$M^{-1} (\et) \cdot c^{- 1 / \bt} x^{1 / \bt} \leq M^{-1} (\et x)$.
Since $\et \in (0, 1]$ and $x \in (0, x_0]$
are arbitrary,
we conclude that $1 / \bt \in S_1 (M^{-1}, x_0)$.
Therefore $1 / \bt \geq \af_1 (M^{-1})$,
so $\bt \leq \af_1 (M^{-1})^{-1}$.
Since $\bt < \af_{0} (M)$ is arbitrary,
the claim follows.

For the reverse inequality,
let $\bt \in (0, \I)$ satisfy $\bt < \af_{1} (M^{-1})^{-1}$.
Then $1 / \bt > \af_{1} (M^{-1})$,
so there is $x_0 \in (0, 1]$
such that $1 / \bt \in S_{1} (M^{-1}, x_0)$,
that is,
there is a constant $d > 0$
such that $M^{-1} (\et x) \geq d M^{-1} (\et) x^{1 / \bt}$
for all $\et \in (0, 1]$ and all $x \in (0, x_0]$.
Set $t_0 = \min \bigl( 1, \, d x_0^{1 / \bt} \bigr)$.
Let $\ld \in (0, 1]$ and let $t \in (0, t_0]$.
Then $M (\ld) \in (0, 1]$ and $d^{- \bt} t^{\bt} \in (0, x_0]$,
so
\[
M^{-1} \bigl(  M (\ld) d^{- \bt} t^{\bt} \bigr)
 \geq d M^{-1} (M (\ld)) \bigl( d^{- \bt} t^{\bt} \bigr)^{1 / \bt}
 = \ld t,
\]
whence
$M (\ld t) \leq d^{- \bt} M (\ld) t^{\bt}$.
This computation shows that
$\bt \leq \af_{0} (M, t_0) \leq \af_{0} (M)$.
\end{proof}

\begin{lem}\label{L_9819_Exp1}
Let $M \in \cF$ and let $\bt_0, \bt_1 \in (0, \I)$
satisfy $\bt_0 < \af_{0} (M)$ and $\af_{1} (M) < \bt_1$.
Then there are $N \in \cF$ and $t_0 \in (0, 1]$
such that $N \sim M$, $N (1) = 1$, and
\[
N (\ld) t^{\bt_1} \leq N (\ld t) \leq N (\ld) t^{\bt_0}
\]
for all $\ld \in (0, 1]$ and all $t \in (0, t_0]$.
If $M$ is convex then $N$ can be chosen to be convex,
and if $M$ is concave then $N$ can be chosen to be concave.
\end{lem}

\begin{proof}
Choose $\gm_0, \gm_1 \in (0, \I)$
such that
\[
\bt_0 < \gm_0 < \af_{0} (M)
\andeqn
\af_{1} (M) < \gm_1 < \bt_1.
\]
Then there are constants $c_0, c_1 \in (0, \I)$
and $t_0 \in (0, 1]$ such that
\[
c_1 M (\ld) t^{\gm_1} \leq M (\ld t) \leq c_0 M (\ld) t^{\gm_0}
\]
for all $\ld \in (0, 1]$ and all $t \in (0, t_0]$.
Choose $t_1 \in (0, 1]$ such that whenever $t \in (0, t_1]$
we have
$t^{\bt_0 - \gm_0} > c_0$ and $t^{\bt_1 - \gm_1} < c_1$.
For $t \in [0, \I)$ define $N_0 (t) = M (t_1 t)$.
Now suppose that $\ld \in (0, 1]$ and $t \in (0, t_0]$.
Using $\ld t_1 \leq 1$ at the second step
and $c_0 t^{\gm_0} \leq t^{\bt_0}$ at the third step,
we get
\[
N_0 (\ld t)
 = M (\ld t_1 t)
 \leq c_0 M (\ld t_1) t^{\gm_0}
 \leq M (\ld t_1) t^{\bt_0}
 = N_0 (\ld) t^{\bt_0}.
\]
Similarly,
using $c_1 t^{\gm_1} \geq t^{\bt_1}$ at the third step,
\[
N_0 (\ld t)
 = M (\ld t_1 t)
 \geq c_1 M (\ld t_1) t^{\gm_1}
 \geq M (\ld t_1) t^{\bt_1}
 = N_0 (\ld) t^{\bt_1}.
\]
The proof is now completed by setting
$N (t) = N_0 (1)^{-1} N_0 (t)$ for $t \in [0, \I)$.
It is obvious that if $M$ is convex then so is~$N$,
and similarly for concavity.
\end{proof}

\begin{lem}\label{L_9825_ExistF0}
Let $t_0, s_0, s_1 \in (0, 1)$,
and suppose that $s_0 \leq s_1$.
Then there exist a strictly increasing
concave $C^{\infty}$~function
$f \colon (0, 1] \to (0, 1]$
and $\ta \in (0, 1)$ such that:
\begin{enumerate}
%
\item\label{Item_9825_ExistF1_f1}
$f (1) = 1$.
\item\label{Item_9825_ExistF1_fprime1}
$f' (1) = s_1$.
\item\label{Item_9825_ExistF1_tauSmall}
$0 < \ta < t_0$.
\item\label{Item_9825_ExistF1_Ftau}
$f (\ta) = \ta^{s_0}$.
\item\label{Item_9825_ExistF1_fprtau}
$f' (\ta) \leq s_1 \ta^{s_1 - 1}$.
\end{enumerate}
\end{lem}

\begin{proof}
If $s_0 = s_1$, take $f (t) = t^{s_1}$ for $t \in [0, 1]$
and take $\ta = t_0 / 2$.
Otherwise, for $\om \in  (s_1, 1]$
define $g_{\om}, h_{\om} \colon [0, 1] \to \R$ by
\[
g_{\om} (t) = 1 + \frac{s_1}{\om} (t^{\om} - 1 )
\andeqn
h_{\om} (t) = g_{\om} (t) - t^{s_0}.
\]
Then $g_{\om}$ satisfies (\ref{Item_9825_ExistF1_f1})
and~(\ref{Item_9825_ExistF1_fprime1}).
Clearly $g_{\om}$ is $C^{\infty}$~on $(0, 1]$.
For $t \in (0, 1]$ we have
$g_{\om}' (t) = s_1 t^{\om - 1} \leq s_1 t^{s_1 - 1}$,
which means that $g_{\om}$ satisfies~(\ref{Item_9825_ExistF1_fprtau})
for any $\ta \in (0, 1]$.
It is obvious that $g_{\om}$ is strictly increasing,
and $g_{\om}$ is concave by the second derivative test.

One checks that $h_{\om}' (t) = 0$
for exactly one value of $t$ in $(0, 1]$, namely
\[
c_{\om} = \exp \left( - \frac{1}{\om - s_0}
       \log \left( \frac{s_1}{s_0} \right) \right).
\]
Moreover, $h_{\om} (1) = 0$, $h_{\om}' (1) = s_1 - s_0 > 0$,
and $h_{\om} (0) = 1 - \frac{s_1}{\om} > 0$.
Therefore $h_{\om}' (t) < 0$ for $t \in (0, c_{\om})$
and $h_{\om}' (t) > 0$ for $t \in (c_{\om}, 1]$,
so there is exactly one number $\ta_{\om} \in (0, 1)$
such that $h_{\om} (\ta_{\om}) = 0$.
Moreover, $0 < \ta_{\om} < c_{\om}$.
One checks that if $s_1 < \om < \zt \leq 1$
then $c_{\om} < c_{\zt} < 1$.

If $c_1 \leq t_0$,
then we must have $\ta_1 \leq t_0$,
and we can take $f = g_1$ and $\ta = \ta_1$.
Otherwise,
set $\dt = \inf_{t \in [t_0, c_1]} | t^{s_0} - t^{s_1} |$.
Then $\dt > 0$.
The functions $g_{\om}$ converge uniformly to $t^{s_1}$
on $[t_0, 1]$ as $\om \to (s_1)^{+}$,
as is seen from the inequality
\[
| g_{\om} - t^{s_1} |
 \leq \left| 1 - \frac{s_1}{\om} \right|
    + \left| 1 - \frac{s_1}{\om} \right| t^{\om}
    + \bigl| \exp \bigl( \om \log (t) \bigr)
          -\exp  \bigl( s_1 \log (t) \bigr) \bigr|.
\]
Therefore there is $\om \in (s_1, 1]$
such that
$| g_{\om} - t^{s_1} | < \frac{\dt}{2}$
for all $t \in [t_0, 1]$.
Then $g_{\om} (t) \neq t^{s_0}$
for $t \in [t_0, c_1]$.
Since $\ta_{\om} < c_{\om} \leq c_1$
and $g_{\om} ( \ta_{\om} ) = \ta_{\om}^{s_0}$,
it follows that $\ta_{\om} < t_0$.
We complete the proof by setting
$f = g_{\om}$ and $\ta = \ta_{\om}$.
Condition~(\ref{Item_9825_ExistF1_tauSmall})
holds by construction,
and (\ref{Item_9825_ExistF1_Ftau}) is $h_{\om} (\ta_{\om}) = 0$.
\end{proof}

\begin{lem}\label{L_9822_Main}
Let $\ph \in \cF$,
and suppose that $\ph$ is concave and
$0 < \af_{0} (\ph) \leq \af_{1} (\ph) < 1$.
Let $\ep > 0$ and $\gm \in (0, \I)$ satisfy
\[
\ep < \af_{0} (\ph)
\andeqn
0 < \frac{\af_{0} (\ph) - \ep}{\gm}
  \leq \frac{\af_{1} (\ph) - \ep}{\gm}
  < 1.
\]
Then there exists a concave function $\ps \in \cF$
such that the function
$t \mapsto t^{\ep} \ps (t)^{\gm}$
is equivalent to~$\ph$ at zero
and such that $\ps (1) = 1$ and $0 < \af_0 (\ps) \leq \af_1 (\ps) < 1$.
\end{lem}

The proof is a modification of the proof of
Proposition 4.c.8 of~\cite{LndTzf1}.
In particular,
the choices made give a sequence
$\et \in \{ 0, 1 \}^{\N}$ as in~\cite{LndTzf1}.
Unfortunately,
the relationship between
the inverses of the functions
$\ps$ and $t \mapsto t^{\ep} \ps (t)^{\gm}$
doesn't seem to be simple enough to use
Proposition 4.c.8 of~\cite{LndTzf1} directly.

\begin{proof}[Proof of Lemma~\ref{L_9822_Main}]
Choose $q_0, q_1 \in (0, 1)$
such that $q_0 < \af_{0} (\ph)$, $\af_{1} (\ph) < q_1$,
and the numbers
\begin{equation}\label{Eq_9822_sDfn}
s_0 = \frac{q_{0} - \ep}{\gm}
\andeqn
s_1 = \frac{q_{1} - \ep}{\gm}
\end{equation}
satisfy $0 < s_0 < s_1 < 1$.
Use Lemma~\ref{L_9819_Exp1}
to choose $t_0 \in (0, 1]$
and a concave function $\ph_0 \in \cF$
such that $\ph_0 \sim \ph$, $\ph_0 (1) =  1$, and
\begin{equation}\label{Eq_9822_Power}
\ph_0 (\ld) t^{q_1} \leq \ph_0 (\ld t) \leq \ph_0 (\ld) t^{q_0}
\end{equation}
whenever $\ld \in (0, 1]$ and $t \in (0, t_0]$.

Define $f_1 \colon (0, 1] \to (0, 1]$
by $f_1 (t) = t^{s_1}$ for $t \in (0, 1]$.
Apply Lemma~\ref{L_9825_ExistF0},
getting $\ta \in (0, t_0)$
and a function from $(0, 1]$ to $(0, 1]$,
which we call~$f_0$.

We now define a function $\ps_0 \colon (0, 1] \to (0, \I)$
and a sequence $\et \in \{ 0, 1 \}^{\N}$
by an inductive procedure,
for $\ps_0$ using pieces which look like $f_0$ or like~$f_1$.
Set $\ps_0 (1) = 1$.
Suppose $n \in \Nz$ and that $\ps_0 (\ta^n)$ has been defined.
We define $\et (n)$ and $\ps_0 (t)$
for $t \in [ \ta^{n + 1}, \, \ta^n )$
as follows.
If
\begin{equation}\label{Eq_9822_ifless}
(\ta^n)^{\ep} \ps_0 (\ta^n)^{\gm} \ta^{ \ep + \gm s_0}
  \leq \ph_0 (\ta^{n + 1}),
\end{equation}
set
\begin{equation}\label{Eq_9822_less}
\et (n) = 0
\andeqn
\ps_0 (t)
 = \ps_0 (\ta^n) f_0 \bigl( \ta^{- n} t \bigr).
\end{equation}
If
\begin{equation}\label{Eq_9822_ifbigger}
(\ta^n)^{\ep} \ps_0 (\ta^n)^{\gm} \ta^{ \ep + \gm s_0}
  > \ph_0 (\ta^{n + 1}),
\end{equation}
set
\begin{equation}\label{Eq_9822_bigger}
\et (n) = 1
\andeqn
\ps_0 (t)
 = \ps_0 (\ta^n) f_1 \bigl( \ta^{- n} t \bigr).
\end{equation}

It is clear that $\ps_0$ is \ct{} and strictly increasing.

We claim that for all $n \in \Nz$ we have
\begin{equation}\label{Eq_9822_DiscEst}
(\ta^n)^{\ep} \ps_0 (\ta^n)^{\gm}
 \leq \ph_0 (\ta^{n})
 \leq \ta^{\gm (s_0 - s_1)} (\ta^n)^{\ep} \ps_0 (\ta^n)^{\gm}.
\end{equation}
(Note that $\ta^{\gm (s_0 - s_1)} > 1$
since $\ta < 1$ and $s_0 - s_1 < 0$.)
The proof is by induction on~$n$.
The relation~(\ref{Eq_9822_DiscEst})
certainly holds for $n = 0$:
it says $1 \leq 1 \leq \ta^{\gm (s_0 - s_1)}$.

So suppose that~(\ref{Eq_9822_DiscEst})
holds for some $n \in \Nz$;
we prove it for $n + 1$.
Assume first that (\ref{Eq_9822_ifless}) holds.
Then $\ps_0 (\ta^{n + 1}) = \ps_0 (\ta^n) \ta^{s_0}$.
Therefore,
using (\ref{Eq_9822_ifless}) at the second step,
$\ta^n \leq 1$, $\ta < t_0$, and (\ref{Eq_9822_Power})
at the third step,
the induction hypothesis at the fourth step,
and (\ref{Eq_9822_sDfn}) at the sixth step,
\begin{align*}
(\ta^{n + 1})^{\ep} \ps_0 (\ta^{n + 1})^{\gm}
& = (\ta^n)^{\ep} \ps_0 (\ta^n)^{\gm} \ta^{\ep} \ta^{\gm s_0}
\\
& \leq \ph_0 (\ta^{n + 1})
  \leq \ph_0 (\ta^{n}) \ta^{q_0}
  \leq \ta^{\gm (s_0 - s_1)} (\ta^n)^{\ep} \ps_0 (\ta^n)^{\gm}
    \ta^{q_0}
\\
& = \ta^{\gm (s_0 - s_1)}
     (\ta^{n + 1})^{\ep} \ps_0 (\ta^{n + 1})^{\gm}
      \ta^{- \ep} \ta^{- \gm s_0} \ta^{q_0}
\\
& = \ta^{\gm (s_0 - s_1)}
    (\ta^{n + 1})^{\ep} \ps_0 (\ta^{n + 1})^{\gm}.
\end{align*}
This is the desired conclusion.

Now assume instead that (\ref{Eq_9822_ifbigger}) holds.
Then $\ps_0 (\ta^{n + 1}) = \ps_0 (\ta^n) \ta^{s_1}$.
Therefore,
using the induction hypothesis at the second step,
$\ta^n \leq 1$, $\ta < t_0$, and (\ref{Eq_9822_Power})
at the third step,
(\ref{Eq_9822_sDfn}) at the fourth step,
and (\ref{Eq_9822_ifbigger}) at the fifth step,
\begin{align*}
(\ta^{n + 1})^{\ep} \ps_0 (\ta^{n + 1})^{\gm}
& = (\ta^n)^{\ep} \ps_0 (\ta^n)^{\gm} \ta^{\ep} \ta^{\gm s_1}
\\
& \leq \ph_0 (\ta^{n}) \ta^{\ep + \gm s_1}
  \leq \ph_0 (\ta^{n + 1}) \ta^{- q_1} \ta^{\ep + \gm s_1}
\\
& = \ph_0 (\ta^{n + 1})
  < (\ta^n)^{\ep} \ps_0 (\ta^n)^{\gm} \ta^{ \ep + \gm s_0}
\\
& = (\ta^{n + 1})^{\ep} \ps_0 (\ta^{n + 1})^{\gm}
    \ta^{- \ep} \ta^{- \gm s_1}
    \ta^{ \ep + \gm s_0}
\\
& = \ta^{\gm (s_0 - s_1)}
   (\ta^n)^{\ep} \ps_0 (\ta^n)^{\gm}.
\end{align*}
This completes the induction.

We next claim that for all $t \in (0, 1]$ we have
\begin{equation}\label{Eq_9822_ContEst}
\ta^{q_1} t^{\ep} \ps_0 (t)^{\gm}
 \leq \ph_0 (t)
 \leq \ta^{\gm (s_0 - s_1)}
   \ta^{- \ep - \gm s_1} t^{\ep} \ps_0 (t)^{\gm}.
\end{equation}
To prove the claim,
choose $n \in \Nz$ such that
$\ta^{n + 1} \leq t \leq \ta^{n}$.
Then $\ps_0 (\ta^{n + 1})$ is either $\ps_0 (\ta^n) \ta^{s_0}$
or $\ps_0 (\ta^n) \ta^{s_1}$,
and in either case
\begin{equation}\label{Eq_9822_s1s0}
\ps_0 (\ta^n) \ta^{s_1}
 \leq \ps_0 (\ta^{n + 1})
 \leq \ps_0 (\ta^n) \ta^{s_0}.
\end{equation}
Since the function $t \mapsto t^{\ep} \ps_0 (t)^{\gm}$
is also strictly increasing,
we have,
using (\ref{Eq_9822_DiscEst}) at the second and sixth steps,
$\ta^n \leq 1$, $\ta < t_0$, and (\ref{Eq_9822_Power})
at the third step,
and (\ref{Eq_9822_s1s0}) at the seventh step,
\begin{align*}
\ta^{q_1} t^{\ep} \ps_0 (t)^{\gm}
& \leq \ta^{q_1} (\ta^{n})^{\ep} \ps_0 (\ta^{n})^{\gm}
  \leq \ta^{q_1} \ph_0 (\ta^{n})
  \leq \ph_0 (\ta^{n + 1})
\\
& \leq \ph_0 (t)
  \leq \ph_0 (\ta^{n})
  \leq \ta^{\gm (s_0 - s_1)} (\ta^n)^{\ep} \ps_0 (\ta^n)^{\gm}
\\
& \leq \ta^{\gm (s_0 - s_1)} \ta^{- \ep} (\ta^{n + 1})^{\ep}
   \ta^{- \gm s_1} \ps_0 (\ta^{n + 1})^{\gm}
\\
& \leq \ta^{\gm (s_0 - s_1) - \ep - \gm s_1}
    t^{\ep} \ps_0 (t)^{\gm}.
\end{align*}
The claim follows.

Define $\ps \colon [0, \I) \to [0, \I)$
by
\[
\ps (t)
 = \begin{cases}
   0         & \hspace*{1em} t = 0
        \\
   \ps_0 (t) & \hspace*{1em} 0 < t \leq 1
       \\
   1 + s_1 (t - 1)
       & \hspace*{1em} 1 < t.
\end{cases}
\]
It follows from (\ref{Eq_9822_s1s0})
and $\ta < 1$ that $\limi{n} \ps (\ta^n) = 0$.
Clearly $\ps (1) = 1$.
Since $\ps_0$ is \ct{} and strictly increasing,
one now easily checks that $\ps$ is \ct{} and strictly increasing.
Also $\ps$ is surjective.
So $\ps \in \cF$.
The previous claim implies that $\ph$ is equivalent at zero
to the function $t \mapsto t^{\ep} \ps (t)^{\gm}$.

We claim that $\ps$ is concave.
One can proceed via the concave analogs of Lemma 4.b.11
and the discussion before Proposition 4.c.4
of~\cite{LndTzf1},
but it is easier to give a direct proof.
The second derivative test applied to $f_1$,
or the choice of $f_0$ using Lemma~\ref{L_9825_ExistF0},
as appropriate,
proves concavity of
the restriction of $\ps$ to each of the intervals
$[ \ta^{n + 1}, \, \ta^n ]$ for $n \in \Nz$
and $[1, \I)$.
Denoting by $D_{-} g (t)$ and $D_{+} g (t)$
the left and right hand derivatives at $t$ of a function~$g$,
it remains only to prove that
for $n \in \Nz$ we have
\[
D_{-} f ( \ta^n) \geq D_{+} f ( \ta^n).
\]

First consider the case $n = 0$.
Here,
regardless of whether $\et (0)$ is $0$ or~$1$,
we have
$D_{-} f (1) = f_{\et (1)}' (1) = s_1 = D_{+} f (1)$.
Now suppose $n \in \N$.
Regardless of the value of $\et (n)$,
we have
%
\begin{align*}
D_{-} f ( \ta^n)
& = \ps_0 (\ta^n) \ta^{- n} f_{\et (n)}' (1)
\\
& = \ps_0 (\ta^n) \ta^{- n} s_1
  =  \ps_0 (\ta^{n - 1}) \ta^{- n + s_{\et (n)}} s_1
  \geq \ps_0 (\ta^{n - 1}) \ta^{- n + s_1} s_1.
\end{align*}
Also,
if $\et (n - 1) = 1$
then $f_{\et (n - 1)}' (\ta) = s_1 \ta^{s_1 - 1}$
by direct computation,
and if $\et (n - 1) = 0$
then $f_{\et (n - 1)}' (\ta) \leq s_1 \ta^{s_1 - 1}$
by Lemma \ref{L_9825_ExistF0}(\ref{Item_9825_ExistF1_fprtau}).
Thus
\[
D_{+} f ( \ta^n)
  = \ps_0 (\ta^{n - 1}) \ta^{- (n - 1)} f_{\et (n - 1)}' (\ta)
  \leq \ps_0 (\ta^{n - 1}) \ta^{- n + s_1} s_1
  \leq D_{-} f ( \ta^n).
\]
This completes the proof of the claim.

We now claim that
$s_0 \leq \af_0 (\ph) \leq \af_1 (\ph) \leq s_1$.
This will imply $0 < \af_0 (\ph) \leq \af_1 (\ph) < 1$,
and finish the proof.

First, by construction,
for every $n \in \Nz$
we have $\ps (\ta^{n + 1}) = \ta^{s_0} \ps (\ta^n)$
or $\ps (\ta^{n + 1}) = \ta^{s_1} \ps (\ta^n)$.
Therefore, for $m, n \in \Nz$,
\begin{equation}\label{Eq_9920_TauEst}
\ta^{s_1 (m - n)}
 \leq \frac{\ps (\ta^m)}{\ps (\ta^n)}
 \leq \ta^{s_0 (m - n)}.
\end{equation}
Now let $t, \ld \in (0, 1]$.
Choose $m, n \in \Nz$ such that
\[
\ta^{m + 1} < \ld t \leq \ta^{m}
\andeqn
\ta^{n + 1} < \ld \leq \ta^{n}.
\]
Then $\ta^{m - n + 1} < t < \ta^{m - n - 1}$,
so, by~(\ref{Eq_9920_TauEst}),
\[
\ta^{2 s_1} t^{s_1}
 \leq \ta^{s_1 (m - n + 1)}
 \leq \frac{\ps (\ta^{m + 1})}{\ps (\ta^n)}
 \leq \frac{\ps (\ld t)}{\ps (\ld)}
 \leq \frac{\ps (\ta^{m})}{\ps (\ta^{n + 1})}
 \leq \ta^{s_0 (m - n - 1)}
 \leq \ta^{- 2 s_0} t^{s_0}.
\]
This shows that $s_0 \in S_0 (\ph, 1)$ and $s_1 \in S_1 (\ph, 1)$,
and the claim follows by Lemma~\ref{L_9818_OrderExp}.
\end{proof}

\begin{proof}[Proof of Proposition~\ref{P_9825_IntermSpace}]
By Proposition~\ref{P_9818_IfRefl}
and Proposition~\ref{P_9817_SameSp},
we may assume $1 < \af_{0} (M) \leq \af_{1} (M) < \I$.
Since $M \sim M (1)^{-1} M$,
we may further assume $M (1) = 1$.
Set $\ph = M^{-1}$.
Then $0 < \af_{0} (\ph) \leq \af_{1} (\ph) < 1$
by Lemma~\ref{L_9819_InvExp}.
Set
\[
\ep = \frac{1}{p} \min \bigl( \af_{0} (\ph),
            \, 1 - \af_{1} (\ph)  \bigr)
\andeqn
\gm = 1 - p \ep.
\]
Then
\[
\frac{\af_{0} (\ph) - \ep}{\gm} > 0
\andeqn
\frac{\af_{1} (\ph) - \ep}{\gm}
 \leq \frac{\af_{1} (\ph) - \ep}{\af_{1} (\ph)}
 < 1.
\]
Therefore we may apply Lemma~\ref{L_9822_Main}
with these choices of $\ep$ and~$\gm$,
getting a concave function $\ps \in \cF$.
Set $N = \ps^{-1}$.
Set $\ps_1 = \ps$,
and define $\ps_0 \colon [0, \I) \to [0, \I)$
by $\ps_0 (t) = t^{ 1/p}$ for $t \in [0, \I)$.
For any $\te \in (0, 1)$,
the function $\ps_{\te}$,
given by $\ps_{\te} (t) = \ps_0 (t)^{1 - \te} \ps_1 (t)^{\te}$
for $t \in [0, \I)$,
is clearly in~$\cF$,
and is concave
by an argument on page 165 of~\cite{Cldn}.
So by Lemma \ref{L_9817_FProp}(\ref{Item_L_9817_FProp_InvOrl}),
the function $\ps_{\te}$
is the inverse of an Orlicz function~$N_{\te}$,
here clearly nondegenerate.
Taking $\te = \gm$,
we have $\ps_{\te} (t) = t^{\ep} \ps_1 (t)^{\gm}$.
Thus $\ps_{\te} \sim \ph$.
So Lemma~\ref{L_9817_Inv} implies $N_{\te} \sim \ph^{- 1} = M$,
as desired.

It remains to prove that $h^N = l^N$.
The construction in Lemma~\ref{L_9822_Main}
gives $\ps (1) = 1$ and $0 < \af_0 (\ps) \leq \af_1 (\ps) < 1$.
So $1 < \af_0 (N) \leq \af_1 (N) < \I$
by Lemma~\ref{L_9819_InvExp}.
Choose any $r \in (\af_1 (N), \, \I)$.
Then there are $t_0 \in (0, 1]$ and $C \in (0, \I)$
such that $N (\ld t) \geq C N (\ld) t^r$
for all $\ld \in (0, 1]$ and $t \in (0, t_0]$.
Choose $n \in \N$ such that $2^{- n} < t_0$.
Then
\[
\sup_{\ld \in (0, \, 2^{-n}]} \frac{N (2 \ld)}{N (\ld)}
 \leq \sup_{\ld \in (0, \, 2^{-n}]} \frac{N (2^n \ld)}{N (\ld)}
 \leq C^{-1} 2^{n r}.
\]
This implies that $N$ satisfies the $\Dt_2$-condition at zero
(Definition 4.a.3 of~\cite{LndTzf1}),
so $h^N = l^N$ by Proposition 4.a.4 of~\cite{LndTzf1}.
\end{proof}

\section{Lorentz spaces}\label{Sec_Other}

In this section,
we consider group algebras on Lorentz sequence spaces
as in Section~4.e of~\cite{LndTzf1}.
These spaces are
mostly quite different from Orlicz sequence spaces,
by Theorems 4.e.2 and 4.e.2$'$ of~\cite{LndTzf1}
and the comment afterwards.
When $G$ is countable and $1 < r < p < \I$,
the Lorentz space $l^{p, r} (G)$
(following Section~1.4 of~\cite{Grfs})
is a Lorentz sequence space.
For such $G$, $p$, and~$r$,
we show that simplicity of $C^*_{\mathrm{r}} (G)$
implies simplicity of the analogous algebra defined on $l^{p, r} (G)$,
and similarly for the unique trace property.
There are many more Lorentz sequence spaces,
and for many of these one expects analogous results,
but treating them seems to require a more general interpolation theorem
and possibly a construction
like that in Section~\ref{Sec_OrliczInt}.

The spaces in the following definition are exactly those
of Definition 4.e.1 of~\cite{LndTzf1}.
The difference is in the notation:
as for Orlicz sequence spaces,
we index sequences using $G$ instead of $\N$.
In~\cite{LndTzf1},
the space is called $d (\bt, p)$.

\begin{dfn}\label{D_9226_LorentzGroupSpace}
Let $\bt = (\bt_n)_{n \in \N}$ be a nonincreasing sequence
in $(0, \I)$
such that
\[
\bt_1 = 1,
\qquad
\limi{n} \bt_n = 0,
\andeqn
\sum_{n = 1}^{\I} \bt_n = \I.
\]
Let $p \in [1, \I)$.
Let $G$ be an infinite countable group.
For any family $\xi = (\xi_g)_{g \in G}$ of complex numbers,
define
\[
\| \xi \|_{\bt, p}
 = \sup \Biggl( \Biggl\{
    \Biggl( \sum_{n = 1}^{\I} | \xi_{\sm (n)} |^p \bt_n \Biggr)^{1/p}
     \colon {\mbox{$\sm$ is a bijection $\N \to G$}}
              \Biggr\} \Biggr).
\]
Then define $d^{\bt, p} (G)$,
the {\emph{Lorentz sequence space of~$G$}}
(with parameters $p$ and~$\bt$),
to be the Banach space consisting
of all $\xi$ such that $\| \xi \|_{\bt, p} < \I$,
with the norm $\| \cdot \|_{\bt, p}$.

In $d^{\bt, p} (G)$, for $g \in G$
we let $\dt_{\bt, p, g}$ be the function
$\dt_{\bt, p, g} (g) = 1$
and $\dt_{\bt, p, g} (h) = 0$ for $h \in G \SM \{ g \}$.
We define the {\emph{left regular representation}}
of $G$ on $d^{\bt, p} (G)$
to be the function $w_{\bt, p} \colon G \to L ( d^{\bt, p} (G) )$
given by, for $\xi = (\xi_g)_{g \in G} \in d^{\bt, p} (G)$
and $h \in G$,
\[
\bigl( w_{\bt, p} (g) \xi \bigr)_h = \xi_{g^{-1} h}.
\]
Let $F^{\bt, p}_{\mathrm{r}} (G) \S L ( d^{\bt, p} (G) )$
be the closed linear span
\[
F^{\bt, p}_{\mathrm{r}} (G) = {\overline{\spn}}
   \bigl( \bigl\{ w_{\bt, p} (g) \colon g \in G \bigr\} \bigr).
\]
Define $\ta_{\bt, p} \colon F^{\bt, p}_{\mathrm{r}} (G) \to \C$
as follows:
if $a (\dt_{\bt, p, 1}) = (\xi_g)_{g \in G}$,
then $\ta_{\bt, p} (a) = \xi_1$.
\end{dfn}

\begin{prp}\label{P_9226_LorentzIfRedGpAlg}
Adopt the notation and assumptions
of Definition~\ref{D_9226_LorentzGroupSpace}.
Then
$\bigl( F^{\bt, p}_{\mathrm{r}} (G),
 \, w_{\bt, p}, \, \ta_{\bt, p} \bigr)$
is a reduced group Banach algebra for~$G$.
\end{prp}

\begin{proof}
The proof is essentially the same
as that of the case $p \in [1, \I)$ of Lemma~\ref{L_8Y02_FprGIsRedAlg}.
We need to know that $w_{\bt, p} (g)$
is an isometry for $g \in G$,
which is immediate,
and that the formula
$\om \bigl( (\xi_g)_{g \in G} \bigr) = \xi_1$
defines a linear functional $\om \colon d^{\bt, p} (G) \to \C$
with $\| \om \| = 1$,
which is easy.
\end{proof}

We won't address general Lorentz sequence spaces here;
instead,
we only consider the following particularly important special case,
which can be treated with a standard interpolation theorem.
Let $(X, \mu)$ be a measure space
and let $\xi \colon X \to \C$ be measurable.
The distribution function
(Definition 1.1.1 of~\cite{Grfs})
$d_{\xi} \colon [0, \I) \to [0, \I]$
is given by
$d_{\xi} (\af)
 = \mu \bigl( \bigl\{ x \in X \colon | \xi (x) | > \af \bigr\} \bigr)$,
and the decreasing rearrangement of $\xi$
(Definition 1.4.1 of~\cite{Grfs})
is the function $\xi^* \colon [0, \I) \to [0, \I]$ given by
$\xi^* (\ld)
 = \inf \bigl(
   \bigl\{ \af > 0 \colon d_{\xi} (\af) \leq \ld \bigr\} \bigr)$.
Then for $p, r \in (0, \I)$,
the space $L^{p, r} (X, \mu)$
(Definition 1.4.6 of~\cite{Grfs})
is the set of all measurable
$\xi \colon X \to \C$ such that the quantity
\[
\| \xi \|_{p, r}
 = \left( \int_0^{\I} \bigl( \ld^{1 / p} \xi^* (\ld) \bigr)^r
     \ld^{-1} \, d \ld \right)^{1 / r}
\]
is finite,
as usual mod equality of functions almost everywhere.
We warn that $\xi \mapsto \| \xi \|_{p, r}$ is usually
not a norm, only a quasinorm,
even when $p, r > 1$.
This space is called the {\emph{Lorentz space}}
with indices $p$ and~$r$.
See Section 1.4.2 of~\cite{Grfs}
for how these spaces relate to the usual
spaces $L^p (X, \mu)$ and to each other,
and for the definition of $L^{p, r} (X, \mu)$ when $p = \I$ or $r = \I$.
We mention just a few facts.
We have $L^{p, p} (X, \mu) = L^p (X, \mu)$,
the space $L^{p, \I} (X, \mu)$ is what is usually called
``weak $L^p$'',
and for fixed $p$ the spaces $L^{p, r} (X, \mu)$ increase
with~$r$.

When $\mu$ is counting measure,
we abbreviate $L^{p, r} (X, \mu)$ to $l^{p, r} (X)$.
The following fact is surely well known,
although we have not found it stated anywhere.

\begin{prp}\label{P_9906_LoretzCtgMs}
Let $p \in (1, \I)$, let $r \in (1, p)$,
and define a sequence $\bt = (\bt_n)_{n \in \N}$
in $(0, \I)$ by
$\bt_n = n^{r / p} - (n - 1)^{r / p}$ for $n \in \N$.
Then $\bt$ satisfies
the condition of Definition~\ref{D_9226_LorentzGroupSpace},
and for any countable group~$G$,
we have $d^{\bt, r} (G) = l^{p, r} (G)$
(that is, they contain exactly the same sequences),
with
\begin{equation}\label{Eq_P_9906_LoretzCtgMs}
\| \xi \|_{p, r} = (p / r)^{1 / r} \| \xi \|_{\bt, r}
\end{equation}
for all $\xi \in d^{\bt, p} (G)$.
\end{prp}

\begin{proof}
That $\bt$ satisfies
the conditions in Definition~\ref{D_9226_LorentzGroupSpace}
is easy.

By Example 1.4.8 of~\cite{Grfs},
the equation~(\ref{Eq_P_9906_LoretzCtgMs})
holds for any $\xi \colon \N \to [0, \I)$ with finite support.
Since both $\| \xi \|_{p, r}$ and $(p / r)^{1 / r} \| \xi \|_{\bt, r}$
depend only on the function $| \xi |$,
it follows that (\ref{Eq_P_9906_LoretzCtgMs})
holds for any $\xi \colon \N \to \C$ with finite support.
Such functions are dense in $l^{p, r} (G)$
by Theorem 1.4.13 of~\cite{Grfs},
and are easily seen to be dense in $d^{\bt, p} (G)$.
The result follows.
\end{proof}

As far as we know,
if $1 < p < r < \I$ then $l^{p, r} (\N)$
is not a Lorentz sequence space as defined in~\cite{LndTzf1}.

\begin{thm}\label{T_9906_LorAlg_Powers}
Let $G$ be a countable group,
taken with the discrete topology,
let $p \in (1, \I)$, and let $r \in (1, p)$.
Let $\bt$ be as in Proposition~\ref{P_9906_LoretzCtgMs},
let
$\bigl( F^{\bt, r}_{\mathrm{r}} (G),
  \, w_{\bt, r}, \, \ta_{\bt, r} \bigr)$
be as in Proposition~\ref{P_9226_LorentzIfRedGpAlg},
and let $( F^{p}_{\mathrm{r}} (G), \, w_p, \, \ta_p)$
be as in Example~\ref{Ex_9217_CStar}.
\begin{enumerate}
%
\item\label{Item_T_9906_LorAlg_Powers}
If $( F^{p}_{\mathrm{r}} (G), \, w_p, \, \ta_p)$
has the Powers property,
then
$\bigl( F^{\bt, r}_{\mathrm{r}} (G),
  \, w_{\bt, r}, \, \ta_{\bt, r} \bigr)$
has the Powers property.
\item\label{Item_T_9906_LorAlg_OneElt}
Let $g \in G$,
and suppose that $( F^{p}_{\mathrm{r}} (G), \, w_p, \, \ta_p)$
has the $g$-Powers property.
Then
$\bigl( F^{\bt, r}_{\mathrm{r}} (G),
  \, w_{\bt, r}, \, \ta_{\bt, r} \bigr)$
has the $g$-Powers property.
\end{enumerate}
\end{thm}

\begin{proof}
Let $F$ denote the vector space of all functions from $\N$ to~$\C$
with finite support,
which we regard as a subspace of $l^{s, t} (\N)$
for all $s, t \in (0, \I)$.
Set
\[
p_0 = \frac{p + r}{2},
\qquad
p_1 = \frac{p (p + r)}{2 r},
\andeqn
\te = \frac{1}{2}.
\]
Then
\[
r < p_0,
\qquad
r < p_1,
\andeqn
\frac{1 - \te}{p_0} + \frac{\te}{p_1} = \frac{1}{p}.
\]

We claim that there is a constant~$R$
such that whenever
$T \colon F \to F$ is linear
and $C_0, C_1 \in [0, \I)$ are constants such that
$\| T \xi \|_{p_j} \leq C_{j} \| \xi \|_{p_j}$
for all $\xi \in F$ and for $j \in \{ 0, 1 \}$,
then
$\| T \xi \|_{p, r} \leq R C_{0}^{1 - \te} C_{1}^{\te} \| \xi \|_{p, r}$
for all $\xi \in F$.
We use Theorem 1.4.19 of~\cite{Grfs}
with $(X, \mu)$ taken to be $\N$ with counting measure.
This theorem applies because,
by Proposition 1.1.6 of~\cite{Grfs},
for any $s \in (0, \I)$ and any $\xi \in L^s (X, \mu)$
we have $\| \xi \|_{s, \I} \leq \| \xi \|_s$.
So, in Theorem 1.4.19 of~\cite{Grfs},
take $R = C_{*} (p_0, p_0, p_1, p_1, 1, r, \te)$.
(The entry~$1$
comes from linearity of~$T$.
The constant~$K$ in~\cite{Grfs}
is from the definition of quasilinearity.)

Given the claim,
the proof of~(\ref{Item_T_9906_LorAlg_Powers}) is the same as
the third paragraph of the proof of Theorem~\ref{T_9905_OrlAlg},
using the previous paragraph
in place of Corollary~\ref{C_9904_InterpForUse}.
The proof of~(\ref{Item_T_9906_LorAlg_OneElt})
is, as before,  essentially the same as the proof
of~(\ref{Item_T_9906_LorAlg_Powers}).
\end{proof}

As a corollary, we get the following theorem.

\begin{thm}\label{T_9906_LorSimplToSimple}
Let $G$ be a countable group,
taken with the discrete topology,
let $p \in (1, \I)$, and let $r \in (1, p)$.
Let $\bt$ be as in Proposition~\ref{P_9906_LoretzCtgMs},
and let $F^{\bt, r}_{\mathrm{r}} (G)$
be as in Definition~\ref{D_9226_LorentzGroupSpace}.
Then:
\begin{enumerate}
%
\item\label{Item_9906_LorSimplToSimple}
If $C^*_{\mathrm{r}} (G)$ is simple
then $F^{\bt, r}_{\mathrm{r}} (G)$ is simple.
\item\label{Item_9906_LorUniqToUniq}
If $C^*_{\mathrm{r}} (G)$ has a unique tracial state,
then $F^{\bt, r}_{\mathrm{r}} (G)$ has a unique unital trace.
\end{enumerate}
\end{thm}

\begin{proof}
For~(\ref{Item_9906_LorSimplToSimple}),
combine
Proposition \ref{P_8Y07_HasPowers}(\ref{Item_8Y07_Fpr}),
Theorem \ref{T_9906_LorAlg_Powers}(\ref{Item_T_9906_LorAlg_Powers}),
and Proposition~\ref{L_6Z16_PPImpSimple}.
For~(\ref{Item_9906_LorUniqToUniq}),
combine
Proposition \ref{P_8Y07_HasgPowers}(\ref{Item_8Y07_HasgPowers_Fpr}),
Theorem \ref{T_9906_LorAlg_Powers}(\ref{Item_T_9906_LorAlg_OneElt}),
and Corollary~\ref{AllgPowers}.
\end{proof}

\end{document}